\documentclass{amsart}
\usepackage[all]{xy}
\date{}
\newtheorem{theorem}{Theorem}[section]
\newtheorem{lemma}[theorem]{Lemma}
\newtheorem{corollary}[theorem]{Corollary}
\newtheorem{proposition}[theorem]{Proposition}

\newtheorem{problem}[theorem]{Problem}
\newtheorem{example}[theorem]{Example}
\newtheorem{claim}[theorem]{Claim}
\theoremstyle{definition}
\newtheorem{definition}[theorem]{Definition}
\newtheorem{remark}[theorem]{Remark}

\newcommand{\IR}{\mathbb R}
\newcommand{\IZ}{\mathbb Z}
\newcommand{\C}{\mathcal C}
\newcommand{\Tau}{\mathcal T}
\newcommand{\IC}{\mathbb C}
\newcommand{\IT}{\mathbb T}
\newcommand{\IN}{\mathbb N}
\newcommand{\N}{\mathcal N}
\newcommand{\E}{\mathcal E}
\newcommand{\U}{\mathcal U}
\newcommand{\V}{\mathcal V}
\newcommand{\Ra}{\Rightarrow}
\newcommand{\diam}{\mathrm{diam}}
\newcommand{\supp}{\mathrm{supp}}
\newcommand{\w}{\omega}
\newcommand{\e}{\varepsilon}
\newcommand{\pr}{\mathrm{pr}}

\newcommand{\add}{\mathrm{add}}
\newcommand{\non}{\mathrm{non}}
\newcommand{\cov}{\mathrm{cov}}
\newcommand{\voc}{\mathrm{voc}}
\newcommand{\cof}{\mathrm{cof}}
\newcommand{\I}{\mathcal I}
\newcommand{\J}{\mathcal J}
\newcommand{\A}{\mathcal A}
\newcommand{\F}{\mathcal F}
\newcommand{\M}{\mathcal M}

\begin{document}
\title[The closed Steinhaus properties of $\sigma$-ideals]{The closed Steinhaus properties\\ of $\sigma$-ideals on topological groups}

\author{Taras Banakh,  Lesia Karchevska, Alex Ravsky}
\address{T.Banakh: Ivan Franko National University of Lviv (Ukraine), and Jan Kochanowski University in Kielce (Poland)}
\email{t.o.banakh@gmail.com}
\address{L.Karchevska: Ivan Franko National University of Lviv (Ukraine)}
\email{crazymats@ukr.net}
\address{A.Ravsky: Pidstryhach Institute for Applied Problems of Mechanics and Mathematics of National Academy of Sciences, Lviv, Ukraine}
\email{oravsky@mail.ru}
\dedicatory{Dedicated to the 130th birthday of Hugo Steinhaus (1887--1972)}

\begin{abstract} We prove that any meager quasi-analytic subgroup of a topological group $G$ belongs to every $\sigma$-ideal $\I$ on $G$ possessing the closed $\pm n$-Steinhaus property for some $n\in\IN$. An ideal $\I$ on a topological group $G$ is defined to have the {\em closed $\pm n$-Steinhaus property} if for any closed subsets $A_1,\dots,A_n\notin\I$ of $G$ the product $(A_1\cup A_1^{-1})\cdots (A_n\cup A_n^{-1})$ is not nowhere dense in $G$. 
Since the $\sigma$-ideal $\mathcal E$ generated by closed Haar null sets in a locally compact group $G$ has the closed $\pm 2$-Steinhaus property, we conclude that each meager quasi-analytic subgroup $H\subset G$ belongs to the ideal $\mathcal E$. For analytic subgroups of the real line this result was proved by Laczkovich in 1998.  We shall discuss  possible generalizations of the Laczkovich Theorem to non-locally compact groups and construct an example of a meager Borel subgroup in $\IZ^\w$ which cannot be covered by countably many closed Haar-null (or even closed Haar-meager) sets. On the other hand, assuming that $\cof(\M)=\cov(\M)=\cov(\N)$ we construct a subgroup $H\subset\IZ_2^\w$ which is meager and Haar null but does not belong to the $\sigma$-ideal $\E$. The construction uses a new cardinal characteristic $\voc^*(\I,\J)$ which seems to be interesting by its own.
\end{abstract}

\maketitle 

\section{Introduction}

By a classical result of S.~Banach \cite{Banach}, any subgroup $H$ with the Baire property in a topological group $G$ is either open or belongs to the $\sigma$-ideal $\M$ of meager sets in $G$.
On the other hand, the classical result of Steinhaus \cite{S20} implies that every Haar measurable subgroup $H$ of a locally compact $G$ is either open or belong to the $\sigma$-ideal $\mathcal N$ of Haar null sets in $G$. Much later, in 1998 Laczkovich \cite{La} unified these two results and proved that each analytic subgroup $H$ of a Polish locally compact group $G$ is either open or belongs to the $\sigma$-ideal $\mathcal E$ generated by closed Haar-null sets in $G$. We recall that a topological space $A$ is {\em analytic} if it is a continuous image of a Polish space. A non-empty family $\I$ of subsets of a set $X$ is called an {\em ideal} on $X$ if $\I$ is closed under taking subsets and finite unions. If $\I$ is closed under countable unions, then $\I$ is called a {\em $\sigma$-ideal}.

Motivated by these classical results, in this paper we shall consider the following problem.

\begin{problem}\label{probmain} Detect $\sigma$-ideals $\I$ on Polish groups $G$ containing all meager analytic subgroups of $G$.
\end{problem}

To answer this problem in Section~\ref{s1} we shall define ideals with the closed $\pm n$-Steinhaus property on topological groups and shall prove that such ideals contain all non-open analytic subgroups and all meager quasi-analytic subgroups of the group. Quasi-analytic spaces are introduced and studied in Section~\ref{s0}. In Section~\ref{s2} we shall construct examples of $\sigma$-ideals distinguishing the closed $\pm n$-Steinhaus properties for various $n$. In Section~\ref{s3} we discuss possible extensions of the Laczkovich theorem to non-locally compact Polish groups and in Section~\ref{s4} will construct an example of a meager Borel (more precisely, $\sigma$-Polish) subgroup $G$ in $\IZ^\w$ which cannot be covered by countably many Haar-null sets in $\IZ^\w$. In the final Section~\ref{s7} given two ideal $\I,\J$ on a group $G$ we study the problem of the existence of a subgroup $H\subset G$ such that $H\in\I\setminus\J$. In particular, assuming that $\cof(\M)=\cov(\M)=\cov(\N)$ we construct a subgroup $H$ in the Cantor cube $2^\w$ which belongs to the family $(\M\cap\N)\setminus \E$ (by the Laczkovich Theorem \cite{La} such subgroup $H$ cannot be analytic). In the construction of $H$ an important role belongs to a new cardinal characteristic $\voc^*(\I,\J)$, for which many natural questions still remain open.

\section{Quasi-analytic and universally $\I$-meager sets}\label{s0}

A topological space $X$ is defined to be
\begin{itemize}
\item {\em Baire} if for any countable family $\U$ of open dense subsets in $X$ the intersection $\bigcap\U$ is dense in $X$;
\item {\em hereditarily Baire} if every closed subspace of $X$ is Baire;
\item {\em quasi-analytic} if $X$ is a continuous image of a hereditarily Baire metrizable separable space.
\end{itemize}

Since  Polish spaces are hereditarily Baire, every analytic space is quasi-analytic. Examples of spaces which are not quasi-analytic are uncountable universally meager spaces.

Following \cite{Zak2}, we call a subset $A$ of a Polish space $Y$ {\em universally meager} if for any continuous nowhere constant map $f:X\to Y$ defined on a Polish space $X$ the preimage $f^{-1}(A)$ is meager in $X$. A map $f:X\to Y$ is called {\em nowhere constant} if for each non-empty open set $U\subset X$ the image $f(U)$ contains more than one point.

This notion is a partial case of the notion of a universally $\I$-meager set defined for any ideal $\I$ as follows.

\begin{definition}\label{d:uIm} Let $\I$ be an ideal on a topological space $X$. A subset $A\subset X$ is called {\em universally $\I$-meager} if for any continuous function $f:Z\to A$ defined on a Baire separable metrizable space $Z$ there exists a non-empty open set $U\subset Z$ with $f(U)\in\I$.
\end{definition}

Definition~\ref{d:uIm} implies the following simple proposition.

\begin{proposition}\label{p:elementary} Let $\I\subset\J$ be two ideals on a topological space $X$.
\begin{enumerate}
\item Each set $A\in\I$ is universally $\I$-meager.
\item If a set $A\subset X$ is universally $\I$-meager, then $A$ is universally $\J$-meager.
\end{enumerate}
\end{proposition}

\begin{proposition}\label{p:uM-M} A subset $A$ of a separable metrizable space $X$ belongs to the ideal $\M$ of meager sets in $X$ if and only if $A$ is universally $\M$-meager in $X$.
\end{proposition}

\begin{proof} The ``only if'' part follows from Proposition~\ref{p:elementary}. To prove the ``if'' part, assume that
the set $A$ is not meager in $X$. Let $U$ be the largest open subset of $A$ which belongs to the $\sigma$-ideal $\M$. The set $U$ is well-defined and is equal to the union of all open subsets of $A$, which are meager in $X$. Let $\bar U$ be the closure of $U$ in $A$. Since $\bar U\setminus U$ is nowhere dense in $A$ and in $X$, the set $\bar U$ is meager in $X$. Then the set $A\setminus \bar U$ is non-meager in $X$. Moreover, by the maximality of $U$, each non-empty open subset $V$ of $A\setminus \bar U$ is not meager in $X$ and hence non-meager in $A$. This implies that the space $A\setminus \bar U$ is Baire.

Observe that for the identity map $f:A\setminus \bar U\to A\subset X$ and any non-empty open set $V$ of $A\setminus \bar U$ the set $f(V)=V$ is not meager in $X$. Consequently, $A$ fails to be universally $\M$-meager.
\end{proof}

\begin{proposition} For a subset $A$ of a Polish space $X$ the following conditions are equivalent:
\begin{enumerate}
\item $A$ is universally meager;
\item $A$ is universally $[X]^{<\w}$-meager for the ideal $[X]^{<\w}$ of finite subsets of $X$;
\item $A$ is universally $[X]^{\le}$-meager for the ideal $[X]^{\le\w}$ of countable subsets of $X$.
\end{enumerate}
\end{proposition}

\begin{proof} $(1)\Ra(2)$. Assume that $A$ is universally meager in $X$. To prove that $A$ is universally $[X]^{<\w}$-meager, fix any continuous map $f:Z\to A$ defined on a Baire separable metrizable space $Z$. By Kuratowski Theorem \cite[3.8]{Ke}, the map $f$ can be extended to a continuous map $\bar f:P\to X$ defined on some Polish space $P$ containing $Z$ as a dense subset. Since $A$ is universally meager and the preimage $\bar f^{-1}(A)\supset Z$ is not meager in $P$, the map $\bar f$ is constant on some non-empty open set $U\subset P$. Then $U\cap Z$ is a non-empty open set in $Z$ such that $f(U\cap Z)\subset\bar f(U)\in[X]^{<\w}$.
\smallskip

The implication $(2)\Ra(3)$ is trivial.
\smallskip

$(3)\Ra(1)$. Assume that $A$ is universally $[X]^{\le \w}$-meager in $X$. To prove that $A$ is universally meager in $X$, we should show that for any nowhere constant continuous map $f:P\to X$ defined on a Polish space $P$ the preimage $f^{-1}(A)$ is meager in $P$. To derive a contradiction, assume that $f^{-1}(A)$ is not meager in $P$. Then there is a non-empty open set $W\subset P$ such that the space $Z=W\cap f^{-1}(A)$ is Baire and dense in $W$. Consider the map $g=f|Z:Z\to A$. Since the set $A$ is universally $[X]^{\le\w}$-meager in $X$, there exists a non-empty open set $U\subset W$ such that the set $C=g(U\cap Z)$ is countable. Since the space $U\cap Z\subset \bigcup_{y\in C}U\cap g^{-1}(y)$ is Baire, for some $y\in C$ the preimage $g^{-1}(y)$ contains some non-empty open set $V\subset U$. The density of $V\cap Z$ in $Z$ guarantees that $f(V)=f(V\cap Z)=g(V\cap Z)=\{y\}$ is a singleton, which means that $f$ is not nowhere constant. But this contradicts the choice of $f$.
\end{proof}

\begin{proposition}\label{p:zak} A quasi-analytic set $A$ is a  topological space $X$ belongs to a $\sigma$-ideal $\I$ on $X$ if and only if $A$ is universally $\I$-meager in $X$.
\end{proposition}

\begin{proof} The ``only if'' part follows from Proposition~\ref{p:elementary}. To prove the ``if'' part, assume that $A$ is universally $\I$-meager in $X$. The space $A$, being quasi-analytic, is the image of a hereditarily Baire metrizable separable space $B$ under a continuous surjective map $f:B\to A$.  Let $\U$ be the family of all open sets $U\subset B$ such that $f(U)\in\I$. Since the space $B$ is hereditarily Lindel\"of, the union $U=\bigcup\U$ equals to the union $\bigcup\V$ of some countable subfamily $\V\subset\U$, which implies that the set $f(U)\in\I$ (we recall that $\I$ is a $\sigma$-ideal).

If $A\notin\I$, then $U\ne A$ and the closed subset $B'=B\setminus U$ of $B$ is not empty. By the maximality of $U$, for any non-empty open set $V\subset B'$ the set $f(V)$ does not belong to the ideal $\I$. Since the space $B$ is hereditarily Baire, the closed subspace $B'$ is Baire. Then the map $f|B':B'\to A$ witnesses that $A$ is not universally $\I$-meager. This contradiction shows that $A\in\I$.
\end{proof}

By \cite{Zak1}, uncountable universally meager sets exist in ZFC. On the other hand,  Miller \cite{Miller} constructed a model of ZFC in which every perfectly (and universally) meager subset $A\subset\IR$ has cardinality $|A|\le\w_1<\w_1=\mathfrak c$. We recall \cite{Mil84} that a set $A\subset \IR$ is called {\em perfectly meager} if for any closed subset $P\subset\IR$ without isolated points the intersection $A\cap P$ is meager in $P$.

By Proposition~\ref{p:uM-M}, each universally $\M$-meager subset of a Polish space $X$ is meager.

\begin{problem} Is it consistent that each $\E$-meager subset of a compact topological group $G$ belongs to the $\sigma$-ideal $\E$ generated by closed Haar null sets in $G$?
\end{problem}

\begin{problem} Is it consistent that each subset $A\subset\IR$  of cardinality $|A|=\mathfrak c$ is quasi-analytic?
\end{problem}

\section{The closed Steinhaus properties of ideals on topological  groups}\label{s1}

In this section we shall introduce various closed Steinhaus properties of ideals on topological groups and shall prove that such ideals contains all meager quasi-analytic subgroups.

An ideal $\I$ on a group $G$ is called {\em invariant} if for any $A\in\I$ and $x,y\in G$ we get $xAy\in\I$; $\I$ is {\em symmetric} if for any $A\in\I$ the set $A^{-1}=\{a^{-1}:a\in A\}$ belongs to $\I$.

For subsets $A_1,\dots,A_n$ of a semigroup $G$ let $A_1\cdots A_n=\{a_1\cdots a_n:a_1\in A_1,\dots,a_n\in A_n\}$ be the product of the sets $A_1,\dots,A_n$ in $G$. If all sets $A_i$ are equal to a fixed subset $A\subset G$, then the product $A_1\cdots A_n$ will be denoted by $A^{\cdot n}$. For a subset $A$ of a group $G$ we put $A^\pm=A\cup A^{-1}$.

In the following definition we introduce two central notions considered in this paper.

\begin{definition}\label{d:pmn} An ideal $\I$ on a topological group $G$ is defined to have the {\em closed $n$-Steinhaus property} \textup{(}resp. the {\em closed $\pm n$-Steinhaus property}\textup{)} if  for any closed subsets $A_1,\dots,A_n\notin\I$ of $G$ their product $A_1\cdots A_n$ \textup{(}resp. $A_1^\pm\cdots A_n^\pm$\textup{)} is not nowhere dense in $G$.
\end{definition}

Observe that each ideal with the closed $n$-Steinhaus property has the closed $\pm n$-Steinhaus property.

The closed $n$-Steinhaus property is a partial case of the closed $\e$-Steinhaus property defined as follows.

\begin{definition} Let $\e=(\e_1,\dots,\e_n)\in\{-1,1\}^n$ for some $n\in\IN$. An ideal $\I$ on a topological group $G$ is defined to have the {\em closed $\e$-Steinhaus property} if  for any closed subsets $A_1,\dots,A_n\notin\I$ of $G$ their $\e$-product $A_1^{\e_1}\cdots A_n^{\e_n}=\{x_1^{\e_1}\cdots x_n^{\e_n}:x_1\in A_1,\dots,x_n\in A_n\}$ is not nowhere dense in $G$.
\end{definition}Observe that the closed $n$-Steinhaus property is equal to the closed $\e$-Steinhaus property for the unique sequence $\e\in\{1\}^n$ consisting of $n$ units.

The following simple proposition follows immediately from the definitions.

\begin{proposition}\label{p:sym} Let $\I$ be an ideal on a topological group and $\e\in\{-1,1\}^n$ for $n\in\IN$.
\begin{enumerate}
\item If $\I$ has the closed $\e$-Steinhaus property, then $\I$ has the closed $\pm n$-Steinhaus property;
\item If $\I$ is symmetric, then $\I$ has the closed $\e$-Steinhaus property if and only if $\I$ has the closed $n$-Steinhaus property.
\end{enumerate}
\end{proposition}

It is clear that the ideal $\M_G$ of meager subsets in a topological group $G$ has the closed $1$-Steinhaus property. The classical Steinhaus theorem \cite{S20} (see also \cite{Stromb}) implies that for a locally compact group $G$ the ideal $\N_G$ of Haar null sets and the $\sigma$-ideal $\mathcal E_G$ generated by closed Haar null subsets in $G$ both have the closed $2$-Steinhaus property.


\begin{theorem}\label{t:qua-pmn} Assume that a $\sigma$-ideal $\I$ on a topological group $G$ has the closed $\pm n$-Steinhaus property for some $n\in\IN$. If subsets $A_1,\dots,A_n\subset G$ are not universally $\I$-meager in $G$, then the set $A_1^\pm\cdots A_n^{\pm}$ is not meager in $G$.
\end{theorem}

\begin{proof} To derive a contradiction, assume that the product $A_1^\pm\cdots A_n^\pm$ is meager in $G$ and hence is contained in the union $\bigcup_{k\in\w}F_k$ of some increasing sequence $(F_k)_{k\in\w}$ of closed nowhere dense sets in $G$.

By our assumption, each set $A_i$ is not universally $\I$-meager in $X$. Then we can find a continuous map $f_i:B_i\to A_i$ defined on a Baire  separable metrizable space $B_i$ such that $f(V_i)\notin \I$ for every non-empty open subset $V_i\subset B_i$.  By \cite[8.44]{Ke}, the product $B=\prod_{i=1}^n B_i$ is a Baire space.

By the continuity of the group operations, for every $k\in\IN$ the set $$E_k=\{(x_1,\dots,x_n)\in \prod_{i=1}^nB_i:\forall (\e_1,\dots,\e_n)\in\{-1,1\}^n\;f_1(x_1)^{\e_1}\cdots f_n(x_n)^{\e_k}\in F_k\}$$ is closed in the Baire space $B=\prod_{i=1}^n B_i$. Since $B\subset \bigcup_{k\in\w}E_k$, for some $k\in\w$ the set $E_k$ contains a non-empty open set $V$ of $B$. Find non-empty open sets $V_1\subset B_1,\dots,V_n\subset B_n$ such that $\prod_{i=1}^nV_i\subset V\subset E_k$ and observe that for their images $C_i=f_i(V_i)\subset G$ we get $C_1^\pm\cdots C_n^\pm\subset F_k$ and by the continuity of the group operations, the product $\bar C_1^\pm\cdots \bar C_n^\pm\subset \bar F_k=F_k$ is nowhere dense.
The choice of the sets $B_1,\dots,B_n$ guarantees that $C_i\notin\I$ and hence $\bar C_i\notin\I$ for all $i\le n$. But this contradicts the closed $\pm n$-Steinhaus property of the ideal $\I$.
\end{proof}

\begin{corollary}\label{c:pmn-uIm} If a $\sigma$-ideal $\I$ on a topological group $X$ has the closed $\pm n$-Steinhaus property for some $n\in\w$, then each meager subgroup $H\subset G$ is universally $\I$-meager.
\end{corollary}

By analogy we can prove the following modification of Theorem~\ref{t:qua-pmn}.

\begin{theorem}\label{t:qua-e} Assume that a $\sigma$-ideal $\I$ on a topological group $G$ has the $\e$-Steinhaus property for some $\e=(\e_1,\dots,\e_n)\in\{-1,1\}^n$, $n\in\IN$.  If subsets $A_1,\dots,A_n\subset G$ are not universally $\I$-meager in $G$, then the set $A_1^{\e_1}\cdots A_n^{\e_n}$ is not meager in $G$.
\end{theorem}

\begin{corollary}\label{c:n-uIm} If a $\sigma$-ideal $\I$ on a topological group $X$ has the closed $n$-Steinhaus property for some $n\in\w$, then each meager subsemigroup $H\subset G$ is universally $\I$-meager.
\end{corollary}

Corollaries~\ref{c:pmn-uIm} and \ref{c:n-uIm} will be combined with the following classical result of Piccard \cite{Piccard} and Pettis \cite{Pettis} on sums of non-meager sets with the Baire property in topological groups. We recall that a subset $A$ of a topological space $X$ has the {\em Baire property} in $X$ if for some open set $U\subset X$ the symmetric difference $(U\setminus A)\cup(A\setminus U)$ is meager in $X$. By Nikod\'ym Theorem \cite[4.9]{Todor}, each analytic set in a regular topological space has the Baire property.

\begin{theorem}[Piccard-Pettis]\label{t:PP} For any non-meager sets $A,B$ with the Baire property in a topological group $G$, the set $AB$ has non-empty interior in $G$ and the set $AA^{-1}$ is a neighborhood of the unit in $G$.
\end{theorem}

Corollary~\ref{c:pmn-uIm} will be used to give the following answer to Problem~\ref{probmain}.

\begin{theorem}\label{t:subgr-pmn} If a $\sigma$-ideal $\I$ on a topological group $G$ has the closed $\pm n$-Steinhaus property for some $n\in\IN$, then $\I$ contains every meager quasi-analytic subgroup of $G$ and every non-open analytic subgroup of $G$.
\end{theorem}

\begin{proof} Assume that some quasi-analytic subgroup $H$ of $G$ does not belong to the ideal $\I$. By Proposition~\ref{p:zak}, $H$ is not universally $\I$-meager and by Corollary~\ref{c:pmn-uIm}, $H$ is not meager. If the subgroup $H$ is analytic, then $H$ has the Baire property by Nikod\'ym Theorem \cite[4.9]{Todor} and the set $H\cdot H\subset H$ has non-empty interior in $G$ by Piccard-Pettis Theorem~\ref{t:PP}.
\end{proof}

By analogy we can derive the following theorem from Corollary~\ref{c:n-uIm}.

\begin{theorem}\label{t:subsemi} If a $\sigma$-ideal $\I$ on a topological group $G$ has the closed $n$-Steinhaus property for some $n\in\IN$, then $\I$ contains every meager quasi-analytic subsemigroup of $G$ and every analytic subsemigroup with empty interior in $G$.
\end{theorem}

We shall say that a $\sigma$-ideal $\I$ on a topological group $G$ has the {\em closed $\pm \w$-Steinhaus property} if for any closed sets $A_n\notin\I$, $n\in\IN$, in $G$ the set $A_1^{\pm}\cdots A_m^{\pm}$ is not nowhere dense in $G$ for some $m\in\IN$. It is clear that each ideal $\I$ with the closed $\pm n$-Steinhaus property for some $n\in\w$ has the closed $\pm \w$-Steinhaus property.

\begin{problem} Does each non-open analytic subgroup $H$ of a Polish group $G$ belong to every invariant $\sigma$-ideal with the closed $\pm \w$-Steinhaus property?
\end{problem}

\section{Constructing ideals possessing various closed Steinhaus properties}\label{s2}

The results of the preceding section motivate a problem of finding natural examples of $\sigma$-ideals possessing the  closed $\pm n$-Steinhaus property. Many examples of such ideals can be constructed as follows.

Given a subset $K$ of a topological group $G$, consider the family $\mathcal B_K^\perp$ of all closed subsets $B\subset G$ such that for every $x,y\in G$ the set $K\cap(xBy)$ is nowhere dense in $K$. Let $\mathcal I^\perp_K$ be the $\sigma$-ideal generated by the family $\mathcal B_K^\perp$. The $\sigma$-ideal $\mathcal I^\perp_K$ consists of subsets of countable unions of sets of the family $\mathcal B^\perp_K$. It is clear that for any set $K$ in a topological group $G$ the $\sigma$-ideal $\I^\perp_K$ is invariant. If $K=K^{-1}$, then the ideal $\I^\perp_K$ is symmetric.

For a Baire subspace $K$ of a topological group $G$ the closed $\pm n$-Steinhaus property of the ideal $\I^\perp_K$ can be characterized as follows.

\begin{proposition}\label{p:chi-pmnnS} Let $K$ be a Baire subspace of a topological group $G$. The $\sigma$-ideal $\I^\perp_K$ has the closed $\pm n$-Steinhaus property for some $n\in\IN$ if and only if for any non-open sets $U_1,\dots,U_n\subset K$ and points\break $x_1,y_1,\dots,x_n,y_n\in G$ the set $(x_1\bar U_1y_1)^{\pm}\cdots (x_n\bar U_ny_n)^{\pm}$ is not nowhere dense in $G$.
\end{proposition}

\begin{proof}  If $\mathcal I^\perp_K$ fails to have the closed $\pm n$-Steinhaus property, then there are closed sets $A_1,\dots,A_n\notin\mathcal I^\perp_K$ in $G$ such that the product $A_1^{\pm}\cdots A_n^{\pm}$ is nowhere dense in $G$. It follows from $A_1,\dots,A_n\notin\mathcal B^\perp_K$ that for every $i\le n$ there are points $x_i,y_i\in G$ such that the set $K\cap (x_i^{-1}A_iy_i^{-1})$ contains a non-empty open subset $U_i$ of $K$. It follows that $\bar U_i\subset x_i^{-1}A_iy_i^{-1}$ and hence  $x_i\bar U_iy_i\subset A_i$. Then the set $(x_1\bar U_1 y_1)^{\pm}\cdots (x_n\bar U_ny_n)^{\pm}\subset A_1^{\pm}\cdots A_n^{\pm}$ is nowhere dense in $G$.
This completes the proof of the ``if'' part.

To prove the ``only if'', assume that for some non-open sets $U_1,\dots,U_n\subset K$ and some points $x_1,y_1,\dots,x_n,y_n\in G$ the set $(x_1\bar U_1y_1)^{\pm}\cdots (x_n\bar U_ny_n)^{\pm}$ is nowhere dense in $G$. We claim that for every $i\le n$ the closed set $A_i=x_i\bar U_i y_i$ does not belong to the ideal $\mathcal I^\perp_K$. Assuming the opposite, we can find a countable subfamily $\F_i\subset\mathcal B^\perp_K$ such that $A_i\subset\bigcup\F_i$. Since the space $K$ is Baire, so is its open subspace $U_i$ and its shift $x_iU_iy_i\subset A_i\subset\bigcup\F_i$. Since $x_iU_iy_i$ is Baire, there is a set $F\in\F_i$ such that the set $x_iU_iy_i\cap F$ contains some open subset $V_i$ of $x_iU_iy_i$. Then $x_i^{-1}V_iy_i^{-1}\subset U_i\subset K\cap x_i^{-1}Fy_i^{-1}$ is a non-empty open subset of $K$, which contradicts the inclusion $F\in\mathcal B^\perp_K$. This contradiction shows that the for every $i\le n$ the closed set $A_i=x_i\bar U_iy_i$ does not belong to the ideal $\I$. Since the product $A_1^{\pm}\cdots A_n^{\pm}=(x_1\bar U_1y_1)^{\pm}\cdots (x_n\bar U_ny_n)^{\pm}$ is nowhere dense in $G$, the ideal $\I$ fails to have the closed $\pm n$-Steinhaus property.
\end{proof}

By analogy we can prove a characterization of $\sigma$-ideals $\I^\perp_K$ possessing the closed $\e$-Steinhaus property.

\begin{proposition}\label{p:chi-nnS} Let $K$ be a Baire subspace of a topological group $G$. The $\sigma$-ideal $\I^\perp_K$ has the closed $\e$-Steinhaus property for some $\e=(\e_1,\dots,\e_n)\in \{-1,1\}^n$, $n\in\IN$, if and only if for any non-open sets $U_1,\dots,U_n\subset K$ and points $x_1,y_1,\dots,x_n,y_n\in G$ the set $(x_1\bar U_1y_1)^{\e_1}\cdots (x_n\bar U_ny_n)^{\e_n}$ is not nowhere dense in $G$.
\end{proposition}

For symmetric subsets $K$ the characterization given in Proposition~\ref{p:chi-nnS} can be simplified. A subset $K$ of a group $G$ is called {\em symmetric} if $K^{-1}=c^{-1}Kc^{-1}$ for some point $c\in G$ (called the {\em center of symmetry} of $K$).

\begin{proposition}\label{p:chi-nnS2} Let $K$ be a Baire symmetric subspace of a topological group $G$. The $\sigma$-ideal $\I^\perp_K$ has the closed $\e$-Steinhaus property for some $\e=(\e_1,\dots,\e_n)\in \{-1,1\}^n$, $n\in\IN$, if and only if for any non-open sets $U_1,\dots,U_n\subset K$ and points $x_1,y_1,\dots,x_n,y_n\in G$ the set $(x_1\bar U_1y_1)\cdots (x_n\bar U_ny_n)$ is not nowhere dense in $G$.
\end{proposition}

For abelian groups the characterizations given in Propositions~\ref{p:chi-pmnnS}--\ref{p:chi-nnS2} can be simplified.

\begin{proposition}\label{p:chi-apmnnS} Let $K$ be a Baire subspace of a topological abelian group $G$. The $\sigma$-ideal $\I^\perp_K$ has the closed $\pm n$-Steinhaus property for some $n\in\IN$ if and only if for any non-empty open sets $U_1,\dots,U_n\subset K$ there is $(\e_1,\dots\e_n)\in\{-1,1\}^n$ such that set $U_1^{\e_1}\cdots U_n^{\e_n}$ is not nowhere dense in $G$.
\end{proposition}

\begin{proposition}\label{p:chi-annS} Let $K$ be a Baire subspace of a topological abelian group $G$. The $\sigma$-ideal $\I^\perp_K$ has the closed $\e$-Steinhaus property for some $\e=(\e_1,\dots,\e_n)\in\{-1,1\}^n$, $n\in\IN$ if and only if for any non-empty open sets $U_1,\dots,U_n\subset K$ the $\e$-product $U_1^{\e_1}\cdots U^{\e_n}_n$ is not nowhere dense in $G$.
\end{proposition}

\begin{proposition}\label{p:chi-annS2} Let $K$ be a Baire symmetric subspace of a topological abelian group $G$. The $\sigma$-ideal $\I^\perp_K$ has the closed $\e$-Steinhaus property for some $\e=(\e_1,\dots,\e_n)\in \{-1,1\}^n$, $n\in\IN$, if and only if for any non-empty open sets $U_1,\dots,U_n\subset K$ the product $U_1\cdots U_n$ is not nowhere dense in $G$.
\end{proposition}

Now we apply Propositions~\ref{p:chi-apmnnS}--\ref{p:chi-annS2} to sets $K$ having product structure.

\begin{proposition}\label{p:prod} Let $A_k\subset G_k$ be non-empty sets in discrete topological abelian groups $G_k$, $k\in\w$. For the closed set $K=\prod_{k\in\w}A_k$ in the topological group $G=\prod_{k\in\w}G_k$ the $\sigma$-ideal $\I^\perp_K$ has the closed $\pm n$-Steinhaus property for some $n\in\IN$ if and only if for some sequence $\e_1,\dots,\e_n\in\{-1,1\}$ and some $k_0\in\w$ the set $A_k^{\e_1}\cdots A_k^{\e_n}$ equals $G_k$ for all $k\ge k_0$.
\end{proposition}

\begin{proof} If for every sequence $\e=(\e_1,\dots,\e_n)\in\{-1,1\}^n$ the set $A_k^\e=A_k^{\e_1}\cdots A_k^{\e_n}$ is not equal to $G_k$ for infinitely many numbers $k$, then the countable product $\prod_{k\in\w}A_k^\e$ is nowhere dense in the topological group $G=\prod_{k\in\w}G_k$ and the closed set $K^{\pm n}=\bigcup_{\e\in \{-1,1\}^n}K^{\e_1}\cdots K^{\e_n}=\bigcup_{\e\in\{-1,1\}^n}\prod_{k\in\w}A_k^\e$ is nowhere dense in $G$. Since $K\notin\I^\perp_K$, the ideal $\I^\perp_K$ fails to have the closed $\pm n$-Steinhaus property.

Now assume that for some sequence $\e_1,\dots,\e_n\in\{-1,1\}$ and some $k_0\in\w$ the set $A_k^{\e_1}\cdots A_k^{\e_n}=G_k$ for all $k\ge k_0$. By Proposition~\ref{p:chi-apmnnS}, to prove that the ideal $\I^\perp_K$ has the closed $\pm n$-Steinhaus property, it suffices to check that for any non-empty subsets $U_1,\dots,U_n\subset K$ the product $U_1^{\e_1}\cdots U_n^{\e_1}$ is not nowhere dense in $G$. By the definition of the Tychonoff product topology on $K=\prod_{k\in\w}A_k$, there are a number $m\ge k_0$ and points $u_0,\dots,u_n\in \prod_{k<m}G_k$ such that $\{u_i\}\times \prod_{k\ge m} A_k\subset U_i$ for all $i\le n$. Taking into account that $A_k^{\e_1}\cdots A_k^{\e_n}=G_k$ for all $k\ge m\ge k_0$, we conclude that $\{u_0\cdots u_n\}\cdot \prod_{k\ge m} G_k\subset U_1^{\e_1}\cdots U_n^{\e_n}$, which means that the latter set has non-empty interior and is not nowhere dense in the topological group $G$.
\end{proof}

By analogy we can prove the following characterization.

\begin{proposition}\label{p:prod2} Let $A_k\subset G_k$ be non-empty sets in discrete topological abelian groups $G_k$, $k\in\w$. For the closed set $K=\prod_{k\in\w}A_k$ in the topological group $G=\prod_{k\in\w}G_k$ the $\sigma$-ideal $\I^\perp_K$ has the closed $\e$-Steinhaus property for some $\e=(\e_1,\dots,\e_n)\in\{-1,1\}^n$, $n\in\IN$, if and only if for some $k_0\in\w$ the set $A_k^{\e_1}\cdots A_k^{\e_n}$ equals $G_k$ for all $k\ge k_0$.
\end{proposition}

For symmetric sets $A_k$ in abelian groups $G_k$ the characterizations given in Propositions~\ref{p:prod} and \ref{p:prod2} can be unified.

\begin{proposition}\label{p:prod3} Let $n\in\IN$ and $A_k\subset G_k$ be non-empty symmetric sets in discrete topological abelian groups $G_k$, $k\in\w$. For the closed set $K=\prod_{k\in\w}A_k$ in the topological group $G=\prod_{k\in\w}G_k$ and the $\sigma$-ideal $\I^\perp_K$ the following conditions are equivalent:
\begin{enumerate}
\item $\I^\perp_K$ has the closed $n$-Steinhaus property;
\item $\I^\perp_K$ has the closed $\pm n$-Steinhaus property;
\item the set $A_k^{\cdot n}=\{a_1\cdots a_n:a_1,\dots,a_n\in A_k\}$ equals $G_k$ for all but finitely many $k$.
\end{enumerate}
\end{proposition}

\begin{proof} The equivalence $(1)\Leftrightarrow(3)$ follows from Proposition~\ref{p:prod2} and $(1)\Ra(2)$ is trivial. So, it remains to check that $(2)\Ra(3)$. Assuming that $\I^\perp_K$ has the closed $\pm n$-Steinhaus property and applying Proposition~\ref{p:prod}, we can find numbers $\e_1,\dots,\e_n\in\{-1,1\}$ and $k_0\in\w$ such that $A_k^{\e_1}\cdots A_k^{\e_n}=G_k$ for all $k\ge k_0$. Since each set $A_k\subset G_k$ is symmetric, we can find an element $c_k\in G_k$ such that $A_k^{-1}=c_k^{-1}A_kc_k^{-1}$. Observe that $G=A_k^{\e_1}\cdots A_k^{\e_n}=A_k^{\cdot n}\cdot c_1^{\e_1-1}\cdots c_n^{\e_n-1}$ and hence $A_k^{\cdot n}=G_k$ for all $k\ge k_0$.
\end{proof}

Now we are ready construct $\sigma$-ideals distinguishing the closed $\pm n$-Steinhaus properties and  the closed $m$-Steinhaus properties for various $n$ and $m$. For $n\in\IN$ by $$C_n=\{z\in\IC:z^n=1\}$$we denote the cyclic group of order $n$.

\begin{example} For every $n\ge 2$ there exists an $F_\sigma$-supported invariant symmetric ideal $\I_n$ on the compact group $C_{2n+1}^\w$ which has the closed $n$-Steinhaus property but fails to have the closed $\pm(n-1)$-Steinhaus property.
\end{example}

\begin{proof} Apply Proposition~\ref{p:prod3} to the symmetric sets $A_k=\{e^{i\varphi}:\varphi \in\{-1,0,1\}\cdot\tfrac\pi{2n+1}\big\}$ in the cyclic groups $G_k=C_{2n+1}$ for $k\in\w$.
\end{proof}

\begin{example}\label{Hai} For some number sequence $(n_k)_{k\in\w}$ the compact group $G=\prod_{k\in\w}C_{n_k}$ contains a closed subset $K$ such that the $\sigma$-ideal $\I^\perp_K$ has the closed $\pm 2$-Steinhaus property but fails to have the closed $n$-Steinhaus property for every $n\in\IN$. Moreover, the symmetric $\sigma$-ideal $\I^\perp_{K^\pm}$ has the
 closed $\pm 2$-Steinhaus property but fails to have the closed $\e$-Steinhaus property for every $\e\in\bigcup_{n=1}^\infty\{-1,1\}^n$.
\end{example}

\begin{proof} By \cite{Haight} or \cite{JWW}, for every $k\in\IN$ there is a number $n_k$ and a subset $A_k$ in the cyclic group $C_{n_k}$ such that $A_kA_k^{-1}=C_{n_k}$ but the set $A_k^{\cdot k}=\{a_1\cdots a_k:a_1,\dots,a_k\in A_k\}$ is not equal to $C_{n_k}$. By Proposition~\ref{p:prod}, for the closed set $K=\prod_{k=1}^\infty A_k$ in the compact topological group $G=\prod_{k=1}^\infty C_{n_k}$ the $\sigma$-ideal $\I^\perp_K$ has the strong $\pm 2$-Steinhaus property but fails to have the strong $n$-Steinhaus property for all $n\in\IN$.

Now consider the symmetric $\sigma$-ideal $\I^\perp_{K^\pm}$. To show that it has the closed $\pm2$-Steinhaus property, fix any closed sets $A_1,A_2\notin\I^\perp_{K^\pm}$ in $G$ and find points $a_1,a_2\in G$ such that $a_1A_1\cap K^\pm$ and $a_1A_2\cap K^\pm$ are not nowhere dense in $K^\pm=K\cup K^{-1}$.
Since $A_k\cap A_{k}^{-1}\ne A_k$ for all $k\ge 2$, the intersection $K\cap K^{-1}$ is nowhere dense in $K$ and $K^{-1}$. Consequently, we can find numbers $\e_1,\e_2\in\{-1,1\}$ and non-empty open sets $U_1,U_2\subset K$ such that $U_i^{\e_i}\subset a_1A_1\cap K^\pm$ for every $i\in\{1,2\}$. By the definition of the Tychonoff product topology on $K=\prod_{k=1}^\infty A_k$, there are $m\in\IN$ and points $u_1,u_2\in \prod_{k=1}^mA_k$ such that $\{u_i\}\times\prod_{k>m}A_k\subset U_i$ for $i\in\{1,2\}$. Then $$\{u_1u_2^{-1}\}\times \prod_{k>m}C_{n_k}=\{u_1u_2^{-1}\}\times\prod_{k>m}A_kA_k^{-1}\subset U_1U_2^{-1}\subset a_1^{\e_1}A_1^{\e_1}A_2^{-\e_2}a_2^{-\e_2}\subset a_1^{\e_1}a_2^{-\e_2}A_1^\pm\cdot A_2^\pm,$$
and hence the set $A_1^\pm\cdot A_2^\pm$ has non-empty interior in $G$. This completes the proof of the closed $\pm 2$-Steinhaus property of the ideal $\I^\perp_{K^\pm}$.

Now we prove that $\I^\perp_{K^\pm}$ fails to have the closed $\e$-Steinhaus property for every $\e\in\{-1,1\}^n$, $n\in\IN$. Since the ideal $\I^\perp_{K^\pm}$ is symmetric, it suffices to check that it fails to have the closed $n$-Steinhaus property (see Proposition~\ref{p:sym}). The choice of the sets $A_k\subset G$ guarantees that for every $n\in\IN$ the set $K^{\cdot n}=\prod_{k=1}^nA_k^{\cdot n}$ is closed and nowhere dense in $G$. Observing that $K$ is not nowhere dense in $K^\pm=K\cup K^{-1}$, we conclude that $K\notin\I^\perp_{K^\pm}$, witnessing that $\I^\perp_{K^\pm}$ fails to have the closed $n$-Steinhaus property.
\end{proof}

\begin{remark}\label{pr2} By \cite{EV}, \cite[\S12.4]{Fal} the real line contains a Borel subgroup $H$ of arbitrary Hausdorff dimension $d\in(0,1)$. This result combined with Theorem~\ref{t:subgr-pmn} implies that for every $d\in(0,1)$ and $n\in\IN$ the $\sigma$-ideal $\I_d$ generated by closed subsets of Hausdorff dimension $\le d$ in $\IR$ does not have the closed $n$-Steinhaus property.
\end{remark}

\section{Ideals of Haar-null and Haar-meager sets in Polish groups}\label{s3}

The classical Steinhaus Theorem \cite{S20} (see also \cite{Stromb}) implies that for every locally compact topological group $G$ the $\sigma$-ideal $\mathcal E_G$ generated by closed sets of Haar measure zero has the closed $2$-Steinhaus property. This allows us to apply Theorem~\ref{t:subsemi} and deduce the following result proved for analytic subgroups of the real line by Laczkovich \cite{La}.

\begin{theorem}\label{Lacz} Every meager quasi-analytic subsemigroup $H$ in a locally compact group $G$ belongs to the $\sigma$-ideal $\mathcal E_G$ and hence can be covered by countably many closed sets of Haar measure zero.
\end{theorem}

It would be interesting to find versions of Theorem~\ref{Lacz} for non-locally compact Polish groups. For this we should find a counterpart of the ideal $\mathcal E_G$. This can be done in six different ways using the notions of [generically] (left, right) Haar null set. For a Polish space $X$ by $\exp(X)$ we shall denote the space of all non-empty compact sets in $X$ endowed with the Vietoris topology, and by $P(X)$  the space of all Borel $\sigma$-additive probability measures on $X$, endowed with the weak-star topology.

\begin{definition} An analytic subset $A$ of a Polish group $G$ is called
\begin{itemize}
\item {\em left-Haar null} if there exists a measure $\mu\in P(G)$ such that $\mu(xA)=0$ for all $x\in G$;
\item {\em right-Haar null} if there exists a measure $\mu\in P(G)$ such that $\mu(Ay)=0$ for all $y\in G$;
\item {\em Haar null} if there exists a measure $\mu\in P(G)$ such that $\mu(xAy)=0$ for all $x,y\in G$;
\item {\em generically left-Haar null} if the set $\{\mu\in P(G):\forall x\in G\;\;\mu(xA)=0\}$ is comeager in $P(G)$;
\item {\em generically right-Haar null} if the set $\{\mu\in P(G):\forall y\in G\;\;\mu(Ay)=0\}$ is comeager in $P(G)$;
\item {\em generically Haar null} if the set $\{\mu\in P(G):\forall x,y\in G\;\;\mu(xAy)=0\}$ is comeager in $P(G)$.
\end{itemize}
\end{definition}
Haar null sets were introduced by Christensen \cite{Christ} who proved the following important fact:

\begin{theorem}[Christensen]\label{ch} If an analytic subset $A$ of a Polish abelian group $G$ is not Haar null, then the set $A^{-1}A$ is a neighborhood of the unit in $G$.
\end{theorem}

Generically Haar null sets were introduced by Dodos \cite{D1}. In \cite{D1} Dodos proved that each analytic generically Haar null subset of a Polish abelian group is meager. By \cite{D2}, a Polish abelian group $G$ is locally compact if and only if each closed Haar null set in $G$ is generically Haar null. The following fact was proved by Dodos in \cite{Dodos}.

\begin{theorem}[Dodos]\label{dod} If an analytic subset $A$ of a Polish group $G$ is not generically left-Haar null, then the set $A^{-1}A$ is not meager in $G$ and the set $A^{-1}AA^{-1}A$ is a neighborhood of the unit in $G$.
\end{theorem}

\begin{corollary}\label{c:dodos} Each non-open analytic subgroup of a Polish group is generically left-Haar null and generically right-Haar null.
\end{corollary}

Next, we consider ``Haar'' modifications of the notion of a meager sets, suggested by Darji \cite{Darji}.

\begin{definition}
We shall say that a subset $A$ of a topological group $G$ is
\begin{itemize}
\item {\em Haar-meager} if there is a continuous map $f:K\to X$ defined on a metrizable compact space $K$ such that for every $x,y\in G$ the set $f^{-1}(xAy)$ is meager in $K$;
\item {\em left-Haar-meager} if there is a continuous map $f:K\to X$ defined on a metrizable compact space $K$ such that for every $x\in G$ the set $f^{-1}(xA)$ is meager in $K$;
\item {\em right-Haar-meager} if there is a continuous map $f:K\to X$ defined on a metrizable compact space $K$ such that for every $y\in G$ the set $f^{-1}(Ay)$ is meager in $K$;
\item {\em generically Haar-meager} if the subset $\{K\in\exp(X):\forall x,y\in G\;\;K\cap(xAy)\in \M_K\}$ is comeager in $\exp(X)$;
\item {\em generically left-Haar-meager} if the subset $\{K\in\exp(X):\forall x\in G\;\;K\cap(xA)\in \M_K\}$ is comeager in $\exp(X)$;
\item {\em generically right Haar-meager} if the subset $\{K\in\exp(X):\forall y\in G\;\;K\cap(Ay)\in \M_K\}$ is comeager in $\exp(X)$.
\end{itemize}
\end{definition}
Haar-meager sets were introduces by Darji \cite{Darji}. The notion of a generically (left, right) Haar-meager set seems to be new (we intend to devote a paper \cite{B?} to studying this new notion). By \cite{Darji}, each Borel (left or right) Haar-meager set in a Polish group $G$ is meager in $G$.

\begin{proposition}\label{Hnul->Hmeg} Each closed (left, right) Haar null set in a Polish group is (left, right) Haar-meager.
\end{proposition}

\begin{proof} Assuming that a closed subset $A$ is a Polish group $G$ is Haar null, find a measure $\mu\in P(G)$ such that $\mu(xAy)=0$ for all $x,y\in G$. Since each Borel $\sigma$-additive measure on a Polish space is Radon (see \cite[7.1.7]{Bog}), there exists a compact set $K\subset G$ with $\mu(K)>0$. Let $U\subset K$ be the largest open subset of measure $\mu(U)=0$ in $K$. Then $\mu(K\setminus U)=\mu(K)>0$ and each non-empty relatively open subset $V$ of $K\setminus U$ has positive measure $\mu(V)$. Replacing $K$ by $K\setminus U$, we can assume that each non-empty open subset of $K$ has positive $\mu$-measure. In this case the compact set $K$ witnesses that the set $A$ is Haar-meager. Indeed, assuming that for some $x,y\in G$ the closed subset $K\cap xAy$ is non-meager in $K$, we conclude that it contains a non-empty open subset $V$ of $K$, which implies that $\mu(xAy)\ge\mu(K\cap xAy)\ge\mu(V)>0$. But this contradicts the choice of the measure $\mu$. By analogy we can prove the left and right versions of this proposition.
\end{proof}

Closed (left, right) Haar-meager sets in Polish groups can be characterized as closed sets which are not (left, right) prethick.

\begin{definition}
A subset $T$ of a topological group $G$ is called
\begin{itemize}
\item {\em thick} if for every compact set $K\subset G$ there are points $x,y\in G$ such that $K\subset xTy$;
\item {\em prethick} if for every compact set $K\subset G$ there is a finite set $F\subset G$ such that $K\subset FTF$;
 \item {\em left thick} if for every compact set $K\subset G$ there is a point $x\in G$ such that $K\subset xT$;
\item {\em left prethick} if for every compact set $K\subset G$ there is a finite set $F\subset G$ such that $K\subset FT$;
 \item {\em right thick} if for every compact set $K\subset G$ there is a point $y\in G$ such that $K\subset Ty$;
\item {\em right prethick} if for every compact set $K\subset G$ there is a finite set $F\subset G$ such that $K\subset TF$.
\end{itemize}
\end{definition}

\begin{proposition}\label{prethick} For a closed subset $A$ of a Polish group $G$ the following conditions are equivalent:
\begin{enumerate}
\item $A$ is (left, right) Haar-meager;
\item there is a compact subset $K\subset G$ such that for every $x,y\in G$ the set $K\cap xAy$ (resp. $K\cap xA$, $K\cap Ay$) is nowhere dense in $K$;
\item $A$ is not (left, right) prethick.
\end{enumerate}
\end{proposition}

\begin{proof} The implication $(2)\Ra(1)$ is trivial. To see that $(1)\Ra(3)$, assume that $A$ is (left, right) Haar meager and find a continuous map $f:K\to G$ defined on a compact metrizable space $K$ such that for every $x,y\in G$ the preimage $f^{-1}(xAy)$ (resp. $f^{-1}(xA)$, $f^{-1}(Ay)$) is meager in $K$. Then for every finite subset $F\subset G$ the preimage $f^{-1}(FAF)=\bigcup_{x,y\in F}f^{-1}(xAy)$ (resp. $f^{-1}(FA)$, $f^{-1}(AF)$) is meager in $K$, which implies that the compact subset $f(K)$ of $G$ is not contained in the set $FAF$ (resp. $FA$, $AF$) and hence witnesses that $A$ is not prethick.
\smallskip

To prove that $(3)\Ra(2)$ we shall prove that the negation of $(2)$ implies that $A$ is (left, right) thick. We shall prove the two-sided version of this implication (the left and right versions can be proved by analogy). Fix a compact set $K\subset G$ and a complete metric $d$ generating the topology of the Polish space $G$. For a point $x\in G$, set $A\subset G$, and positive number $\e$ put $B(x;\e)=\{g\in G:d(g,x)<\e\}$ and $B(A;\e)=\bigcup_{a\in A}B(a;\e)$.

For every $n\in\w$ choose a finite cover $\C_n$ of $K$ by non-empty closed sets of $d$-diameter $\le 2^{-n}$. In each set $C\in\C_n$ pick up a point $p_C$. For every $n\in\w$ consider the compact set $\sigma_n=\bigcup_{m\ge n}\bigcup_{C\in\C_m}p_C^{-1}C$ and observe that the sequence $(\sigma_n)_{n\in\w}$ converges to the unit $1_G$ of $G$ in the sense that each neighborhood $U\subset G$ of $1_G$ contains all but finitely many sets $\sigma_n$, $n\in\w$. By induction we shall construct a sequence $(\Sigma_n)_{n\in\w}$ of non-empty compact sets in $K$ satisfying the following conditions for every $n\in\w$:
\begin{itemize}
\item[$(\dag_n)$] $\Sigma_{n+1}\subset B(\Sigma_n;2^{-n})$;
\item[$(\ddag_n)$] any open set $U\subset \Sigma_{n+1}$ meeting the set $\Sigma_n$ contains the set $x\cdot \sigma_l$ for some $x\in \Sigma_n$ and some $l\in\w$.
\end{itemize}
We start the inductive construction letting $\Sigma_{-1}=\emptyset$ and $\Sigma_0=\{1_G\}$. Assume that for some $n\in\w$ the set $\Sigma_{n}$ has been constructed. Fix a countable dense set $\{x_m\}_{m\in\w}$ in $\Sigma_n$. For every $m\in\w$ find a number $l(m)\in\w$ such that $x_m\cdot \sigma_{l(m)}\subset B(x_m;2^{-n-m})$. Such number $l(m)$ exists since the sequence $(\sigma_l)_{l\in\w}$ converges to $1_G$. Then $\Sigma_{n+1}=\Sigma_n\cup\bigcup_{m\in\w}x_m\sigma_{l(m)}$ is compact and satisfies the conditions $(\dag_{n+1})$ and $(\ddag_{n+1})$.

After completing the inductive construction, consider the closure $\bar \Sigma$ of the union $\Sigma=\bigcup_{n\in\w}\Sigma_n$. The conditions $(\dag_n)$, $n\in\w$, guarantee that the set $\Sigma$ is totally bounded in the complete metric space $(G,d)$ and hence $\bar \Sigma$ is compact. By the negation of (2), there exist points $x,y\in G$ such that the set $\bar \Sigma\cap xAy$ is not nowhere dense in $\bar\Sigma$ and hence contains a non-empty open subset $U$ of $\bar\Sigma$. By the conditions $(\ddag_n)$, $n\in\w$, the set $U$ contains a shift $z\sigma_l$ of the set $\sigma_l$ for some $l\in\w$. Then $$
xAy\supset U\supset z\sigma_l\supset\bigcup_{C\in\C_l}zp_C^{-1}C$$and
$K=\bigcup\C_l\subset \bigcup_{C\in\C_l}p_Cz^{-1}xAy\subset FAF$ for the finite set $F=\{y\}\cup\{p_Cz^{-1}x:C\in\C_l\}$, witnessing that the set $A$ is prethick.
\end{proof}

The following Steinhaus-type property of Haar-meager sets was proved by Jablo\'nska \cite{Jab}.

\begin{theorem}[Jablo\'nska] If a Borel subset $A$ of a Polish abelian group $G$ is not Haar-meager, then $A^{-1}A$ is a neighborhood of zero.
\end{theorem}

Generically left Haar-meager sets have a weaker property.

\begin{theorem}\label{t:gm} If an (analytic) subset $A$ of a Polish group $X$ is not generically left Haar-meager, then $A^{-1}A$ is not meager in $X$ (and hence $A^{-1}AA^{-1}A$ is a neighborhood of the unit in $X$).
\end{theorem}

\begin{proof} To derive a contradiction, assume that the set $A^{-1}A$ is meager in $X$. Then we can find a meager $F_\sigma$-set $F\subset X$ containing $A^{-1}A$. Consider the continuous map $\mu:X\times X\to X$, $\mu:(x,y)\mapsto x^{-1}y$, and observe that it is open. This implies that $\mu^{-1}(F)$ is a meager $F_\sigma$-set in $X\times X$ and its complement $G=(X\times X)\setminus f^{-1}(F)$ is a dense $G_\delta$-set in $X\times X$.
For a subset $K\subset X$ put $(K)^2=\{(x,y)\in K\times K:x\ne y\}\subset X\times X$.
By Mycielski-Kuratowski Theorem \cite[19.1]{Ke}, the set $H=\{K\in\exp(X):(K)^2\subset G\}$ is a dense $G_\delta$-set in the hyperspace $\exp(X)$. Since the Polish group $G$ contains a non-empty meager subset, it is not discrete and hence contains no isolated points.  By \cite[4.31]{Ke}, the subset $P\subset \exp(X)$ consisting of compact subsets without isolated points is a dense $G_\delta$-set in $X$. Then $H\cap P$ is a dense $G_\delta$-set in $\exp(X)$. We claim that for each compact subset $K\in H\cap P$ and every $x\in A$ the set $K\cap xA$ contains at most one point and hence is meager in $K$.
Assuming that $K\cap xA$ contains two distinct points $xa,xb$, we conclude that $\mu(xa,xb)=a^{-1}b\in A^{-1}A$ and hence $(xa,xb)\in (K)^2\cap f^{-1}(F)$, which contradicts the choice of $K\in H$.
Therefore $H\cap P\subset \{K\in\exp(X):\forall x\in G\;\;K\cap xA\in\M_K\}$, which implies that $A$ is generically left Haar-meager. But this contradicts the choice of $A$.

Now assume that $A$ is analytic. Then the set $A^{-1}A$ is analytic and non-meager in the Polish group $G$. By Piccard-Pettis Theorem \cite[9.9]{Ke}, $A^{-1}AA^{-1}A$ is a neighborhood of the unit in $G$.
\end{proof}

Theorems~\ref{dod} and \ref{t:gm} imply the following partial answer to Problem~\ref{probmain}.

\begin{corollary}\label{c:my} Every analytic non-open subgroup $H$ of a Polish group $G$ is generically left-Haar-meager and generically right Haar-meager in $G$.
\end{corollary}

\section{Three counterexamples}\label{s4}

In this section we present three counterexamples to possible extensions of Theorem~\ref{Lacz}.
First of them will show that in contrast to  Theorem~\ref{Lacz},  Corollaries~\ref{c:dodos} and \ref{c:my} (treating analytic subgroups) are not true for  subsemigroups of Polish groups.

\begin{example}\label{ex1} The closed nowhere dense subsemigroup $\IR_+^\w$ is thick in the Polish abelian group $\IR^\w$ and hence $H$ is neither Haar-null nor Haar-meager in $\IR^\w$.
\end{example}

The second our example shows that ``two-sided'' versions Corollaries~\ref{c:dodos} and \ref{c:my} are false.

\begin{example}\label{ex2} There exists a Polish group $G$ containing a closed nowhere dense thick subgroup $H$. This subgroup is neither Haar-null nor Haar-meager. On the other hand, $H$ is generically left (and right) Haar-null and generically left (and right) Haar-meager.
\end{example}

\begin{proof} Let $X$ be a countable set and $Z\subset X$ be a proper countable subset of $X$.
Let $FS_X$ be the discrete group of finitely supported permutations of $X$ and $FS_Z$ be the subgroup of $FS_X$ consisting of permutations $f:X\to X$ with finite support $\supp(f)=\{x\in X:f(x)\ne x\}\subset Z$. Consider the Polish group $G=FS_X^\w$ and closed nowhere dense subgroup $H=FS_Z^\w$ of $G$. It is easy to show that for any finite subset $A\subset FS_X$ there is a permutation $g\in FS_X$ such that $gAg^{-1}\subset FS_Z$. This means that the subgroup $FS_Z$ is thick in $FS_X$. This property implies that the countable product $H=FS_Z^\w$ is thick in the Polish group $G=FS_Z^\w$. By Propositions~\ref{prethick} and \ref{Hnul->Hmeg}, the closed subgroup $H$ is neither Haar-meager nor Haar-null in $G$.

By Corollaries~\ref{c:dodos} and \ref{c:my} the nowhere dense subgroup $H$ is generically left (and right) Haar-null and generically left (and right) Haar-meager in $G$.
\end{proof}

Our final example shows that Theorem~\ref{Lacz} is not true for non-locally compact Polish groups.
A topological space $X$ is defined to be {\em $\sigma$-Polish} if $X$ can be written as a countable union $X=\bigcup_{n\in\w}X_n$ of closed Polish subspaces of $X$. It is easy to see that each $\sigma$-Polish subspace $X$ of a Polish space $Y$ can be written as the difference $A\setminus B$ of two $F_\sigma$-sets in $Y$. Consequently, $X$ is of Borel classes $F_{\sigma\delta}$ and  $G_{\delta\sigma}$.

\begin{example}\label{e:counter} The Polish group $G=\IZ^\w$ contains a meager $\sigma$-Polish subgroup $H\subset\{0\}\cup\{z\in \IZ^\w:|z^{-1}(0)|<\w\}$ which cannot be covered by countably many closed Haar-meager sets in $G$. The subgroup $H$ is generated by a $G_\delta$-subset $P\subset \IZ^\w$ such that for any non-empty relatively open set $U\subset P$ its closure $\bar U$  is prethick in $G$.
\end{example}

\begin{proof} In the construction of the $G_\delta$-set $P$ we shall use the following elementary lemma.

\begin{lemma}\label{l1} There exists an infinite family $\Tau$ of thick subsets of $\IZ$ and an increasing number sequence $(\Xi_m)_{m\in\w}$ such that for any positive numbers $n\le m$, non-zero integer numbers $\lambda_1,\dots,\lambda_n\in[-m,m]$, pairwise distinct sets $T_1,\dots,T_n\in\Tau$ and points $x_i\in T_i$, $i\le n$, such that $\{x_1,\dots,x_n\}\not\subset [-\Xi_m,\Xi_m]$ the sum\break $\lambda_1x_1+\dots+\lambda_n x_n$ is not equal to zero.
\end{lemma}

\begin{proof} For every $m\in\w$ let $\xi_m\in\w$ be the smallest number such that $2^{2^{x}}-x>m^2(2^{2^{x-1}}+x)$ for all $x\ge \xi_m$, and put $\Xi_m=2^{2^{\xi_m}}+\xi_m$.  Choose an infinite family $\A$ of pairwise disjoint infinite subsets of $\IN$ and for every $A\in\A$ consider the thick subset $T_A=\bigcup_{a\in A}[2^{2^a}-a,2^{2^a}+a]$ of $\IZ$. Here by $[a,b]$ we denote the segment $\{a,\dots,b\}$ of integers. We claim that the family $\Tau=\{T_A\}_{A\in\A}$ and the sequence $(\Xi_m)_{m\in\w}$ have the required property.

Take any positive numbers $n\le m$, non-zero-integer numbers $\lambda_1,\dots,\lambda_n\in[-m,m]$, pairwise distinct sets $T_1,\dots,T_n\in\Tau$ and points $x_1\in T_1,\dots,x_n\in T_n$ such that $\{x_1,\dots,x_n\}\not\subset [-\Xi_m,\Xi_m]$. For every $i\le n$ find an integer number $a_i$ such that $x_i=2^{2^{a_i}}+\e_i$ for some $\e_i\in[-a_i,a_i]$. Since the family $\A$ is disjoint and the sets $T_1,\dots,T_n$ are pairwise distinct, the points $a_1,\dots,a_n$ are pairwise distinct, too. Let $j$ be the unique number such that $a_j=\max\{a_i:1\le i\le n\}$. Taking into account that $\{x_1,\dots,x_n\}\not\subset [0,\Xi_m]=[0,2^{2^{\xi_m}}+\xi_m]$, we conclude that $2^{2^{a_j}}+a_j\ge 2^{2^{a_j}}+\e_j=x_j>2^{2^{\xi_m}}+\xi_m$ and hence $a_j>\xi_m$. The definition of the number $\xi_m$ guarantees that $2^{2^{a_j}}-a_j>m^2(2^{2^{a_j-1}}+a_j)$ and hence
$$|\lambda_jx_j|\ge x_j=2^{2^{a_j}}+\e_j\ge 2^{2^{a_j}}-a_j>m^2(2^{2^{a_j-1}}+a_j)> \sum_{i\ne j}|\lambda_i|(2^{2^{a_i}}+a_i)\ge |\sum_{i\ne j}\lambda_ix_i|,$$and hence $\sum_{i=1}^n\lambda_ix_i\ne0$.
\end{proof}

Now we are ready to start the construction of the $G_\delta$-set $P\subset\IZ^\w$. This construction will be done by induction on the tree $\w^{<\w}=\bigcup_{n\in\w}\w^n$ consisting of finite sequences $s=(s_0,\dots,s_{n-1})\in\w^n$ of finite ordinals. 
For a sequence $s=(s_0,\dots,s_{n-1})\in \w^n$ and a number $m\in\w$ by $s\hat{\;}m=(s_0,\dots,s_{n-1},m)\in \w^{n+1}$ we denote the {\em concatenation} of $s$ and $m$.

For an infinite sequence $s=(s_n)_{n\in\w}\in\IZ^\w$ and a natural number $l\in\w$ let $s|l=(s_0,\dots,s_{l-1})$ be the restriction of the function $s:\w\to\IZ$ to the subset $l=\{0,\dots,l-1\}$.
Observe that the topology of the Polish group $\IZ^\w$ is generated by the ultrametric $$d(x,y)=\inf\{2^{-n}:n\in\w,\;x|n=y|n\},\;\;x,y\in\IZ^\w.$$Observe also that for every $z\in\IZ^\w$ and $n\in\w$ the set $U(z|n)=\{x\in\IZ^\w:x|n=z|n\}$ coincides with the closed ball $\bar B(z;2^{-n})=\{x\in\IZ^\w:d(x,z)\le 2^{-n})$ centered at $z$.

Using Lemma~\ref{l1}, choose a number sequence $(\Xi_m)_{m\in\w}$ and a sequence $(T_s)_{s\in\w^{<\w}}$ of thick sets in the discrete group $\IZ$ such that for every positive integer numbers $n\le m$, finite set $F\subset \w^{<\w}$ of cardinality $|F|\le n$, function $\lambda:F\to [-m,m]\setminus\{0\}$, and numbers $z_s\in T_s$, $s\in F$, such that $\{z_s\}_{s\in F}\not\subset[-\Xi_m,\Xi_m]$ the sum $\sum_{s\in F}\lambda(s)\cdot z_s$ is not equal to zero.

For every $s\in\w^{<\w}$ and $n\in\w$ choose any point $t_{s,n}\in T_s\setminus[-n,n]$ and observe that the set $T_{s,n}=T_s\setminus(t_{s,n}\cup[-n,n])$ remains thick in $\IZ$.

By induction on the tree $\w^{<\w}$ we shall construct a sequence $(z_s)_{s\in\w^{<\w}}$ of points of $\IZ^\w$ and a sequence $(l_s)_{s\in\w^{<\w}}$ of finite ordinals satisfying the following conditions for every $s\in \w^{<\w}$:
\begin{itemize}
\item[$(1_s)$] $U(z_{s\hat{\;} i}|l_{s\hat{\;} i})\cap U(z_{s\hat{\;} j}|l_{s\hat{\;} j})=\emptyset$ for any distinct numbers $i,j\in\w$;
\item[$(2_s)$] $l_{s\hat{\;} i}> l_s+i$ for every $i\in\w$;
\item[$(3_s)$] the closure of the set $\{z_{s\hat{\;}i}\}_{i\in\w}$ contains the prethick set $\mathbb T_s=\{z_s|l_s\}\times\prod_{n\ge l_s}T_{s,n}$ and is contained in the set $\mathbb T^*_s=\{z_s|l_s\}\times \prod_{n\ge l_s}(T_{s,n}\cup\{t_{s,n}\})\subset U(z_s|l_s)$.
\end{itemize}
We start the inductive construction letting $z_0=0$ and $l_0=0$. Assume that for some $s\in\w^{<\w}$ a point $z_s\in\IZ^\w$ and a number $l_s\in\w$ have been constructed. Consider the prethick sets $\mathbb T_s$ and $\IT_s^*$ defined in the conditions $(2_s)$ and $(3_s)$. Since $\mathbb T_s$ is nowhere dense in $\mathbb T^*_s$, we can find a sequence  $(z_{s\hat{\;}i})_{i\in\w}$ of pairwise distinct points of $\mathbb T^*_s$ such that the space $D_s=\{z_{s\hat{\;}i}\}_{i\in\w}$ is discrete and contains $\mathbb T_s$ in its closure. Since $D_s$ is discrete, for every $i\in\w$ we can choose a number $l_{s\hat{\;}i}>l_s+i$ such that the open sets $U(z_{s\hat{\;}i}|l_{s\hat{\;}i})$, $i\in\w$, are pairwise disjoint. Observing that the sequences $(z_{s\hat{\;}i})_{i\in\w}$ and $(l_{s\hat{\;}i})_{i\in\w}$ satisfy the conditions $(1_s)$--$(3_s)$, we complete the inductive step.

We claim that the $G_\delta$-subset $P=\bigcap_{n\in\w}\bigcup_{s\in\w^s}U(z_s|l_s)$ of $\IZ^\w$ has the required properties. First observe that the map $h:\w^\w\to P$ assigning to each infinite sequence $s\in\w^\w$ the unique point $z_s$ of the intersection $\bigcap_{n\in\w}U(z_{s|n}|l_{s|n})$ is a homeomorphism of $\w^\w$ onto $P$. Then the inverse map $h^{-1}:P\to\w^\w$ is a homeomorphism too.

\begin{claim} For every non-empty open set $U\subset P$ its closure $\bar U$ in $\IZ^\w$ is prethick in $\IZ^\w$.
\end{claim}

\begin{proof} Given any non-empty open set $U\subset P$, pick any point $p\in U$ and find a unique infinite sequence $t\in\w^\w$ such that $\{p\}=\bigcap_{m\in\w}U(z_{t|m}|l_{t|m})$. Since the family $\{U(z_{t|m}|l_{t|m})\}_{m\in\w}$ is a neighborhood base at $p$, there is $m\in\w$ such that $U(z_{t|m},l_{t|m})\subset U$. Consider the finite sequence $s=t|m$. The Baire Theorem guarantees that for every $i\in\w$ the intersection $P\cap U(z_{s\hat{\;}i}|l_{s\hat{\;}i})$ contains some point $y_{s\hat{\;}i}$. Taking into account that $$d(z_{s\hat{\;}i},y_{s\hat{\;}i})\le\diam\, U(z_{s\hat{\;}i}|l_{s\hat{\;}i})\le 2^{-l_{s\hat{\;}i}}\le 2^{-i}$$ and $\mathbb T_s$ is contained in the closure of the set $\{z_{s\hat{\;}i}\}_{i\in\w}$, we conclude that the prethick set $\IT_s$ is contained in the closure of the set $\{y_{s\hat{\;}i}\}_{i\in\w}\subset P\cap U(z_s|l_s)\subset U$, which implies that $\bar U$ is prethick.
\end{proof}

\begin{claim} The subgroup $H\subset\IZ^\w$ generated by $P$
cannot be covered by countably many closed Haar-meager sets in $G$.
\end{claim}

\begin{proof} To derive a contradiction, assume that $H\subset\bigcup_{n\in\w}H_n$ where each set $H_n$ is closed and Haar-meager in $\IZ^\w$. Since the Polish space $P=\bigcup_{n\in\w}P\cap H_n$ is Baire, for some $n\in\w$ the set $P\cap H_n$ contains a non-empty open subset $U$ of $P$. Taking into account that closure $\bar U_n\subset H_n$ is prethick in $G$, we conclude that the set $H_n$ is prethick and not Haar-meager (according to Proposition~\ref{prethick}).
\end{proof}

It remains to prove that the subgroup $H\subset\IZ^\w$ generated by the $G_\delta$-set $P$ is $\sigma$-Polish and meager in $\IZ^\w$. This is the most difficult (in technical respect) part of the proof.

We recall that by $h:\w^\w\to P$ we denote the homeomorphism assigning to each infinite sequence $s\in\w^\w$ the unique point $z_s$ of the intersection $\bigcap_{m\in\w}U(z_{s|m}|l_{s|m})$.

For every $n\in\w$ denote by $R_n:\w^\w\to \w^n$, $R_n:s\mapsto s|n$, the restriction operator.
For any finite subset $F\subset \w^\w$ denote by $\delta_F\in\w$ the smallest number such that the restriction $R_n|F$ is injective. Let also $\Lambda_F=\max\{l_{s|\delta_F}:s\in F\}$. Taking into account that $l_{s|m}\ge m$ for every $s\in\w^\w$ and $m\in\w$, we conclude that $\Lambda_F\ge\delta_F$ for every finite subset $F$ of $P$.

\begin{claim}\label{cl2} For every $m\in\IN$, non-empty finite set $F\subset \w^\w$ of cardinality $|F|\le m$ and function $\lambda:F\to[-m,m]\setminus\{0\}$ we get $$\sum_{s\in F}\lambda(s)\cdot  z_s\in\big\{y\in\IZ^\w:\forall k> \max\{\Lambda_F,\Xi_m\}\;\;y(k)\ne 0\big\}.$$
\end{claim}

\begin{proof} Given any number $k>\max\{\Lambda_F,\Xi_m\}$, consider the projection $\pr_k:\IZ^\w\to \IZ$, $\pr_k:(x_i)_{i\in\w}\mapsto x_k$, to the $k$-th coordinate. The claim will be proved as soon as we check that the element $y=\sum_{s\in F}\lambda(s)\cdot  z_s$ has non-zero projection $\pr_k(y)$. Observe that $\pr_k(y)=\sum_{s\in F}\lambda(s)\cdot\pr_k(z_s)$. For every $s\in F$ the equality  $\{z_s\}=\bigcap_{m\in\w}U(z_{ s_p|m}|l_{ s_p|m})$ implies that
$\pr_k(z_s)=\pr_k(z_{ s_p|m})$ for any $m\in\w$ such that $l_{ s_p|m}>k$.

 Taking into account that  $k>\Lambda_F=\max_{s\in F}l_{s|\delta_F}$, for every $s\in F$ we can find a number $m_s\ge\delta_F$ such that $k\in[l_{s|m_s},l_{s|(m_s+1)})$.
Then $$\pr_k(z_s)=\pr_k(z_{s|(m_s+1)}|l_{s|(m_s+1)})\in T_{s|m_s,k}\cup\{t_{s|m_s,k}\}
\subset T_{s|m_s}\setminus [-k,k]\subset T_{s|m_s}\setminus [-\Xi_m,\Xi_m]$$ by the condition $(3_{ s|m_s})$ of the inductive construction.

The definition of the number $\delta_F$ guarantees that the sequences
$s|\delta_F$, $s\in F$, are pairwise distinct and so are the sequences $ s|m_s$, $s\in F$.
Then $\pr_k(y)=\sum_{s\in F}\lambda(s)\pr_k(z_s)\ne0$ by the choice of the family $(T_s)_{s\in\w^{<\w}}$.
\end{proof}

For every $n\in\w$ and a finite sequence $s\in\w^n$ consider the basic closed-and-open subset
$V_s=\{v\in\w^\w:v|n=s\}$ in $\w^\w$ and observe that $h(V_s)=P\cap U(z_s|l_s)$.
For any non-empty finite set $F\subset \w^\w$ put $$V_F=\prod_{s\in F}V_{s|\delta_F}.$$ Given any  function $\lambda:F\to\IZ\setminus\{0\}$ we shall prove that the function
$$\Sigma_\lambda:V_F\to H,\;\;\Sigma_\lambda:v\mapsto \sum_{s\in F}\lambda(s)\cdot z_{v(s)},$$is a closed topological embedding. This will be done in three steps. Let $\|\lambda\|=\max_{s\in F}|\lambda(s)|$.

\begin{claim} The function $\Sigma_\lambda:V_F\to H$ is injective.
\end{claim}

\begin{proof} Choose any distinct sequences $u,v\in V_F$ and consider the symmetric difference $D$ of the sets $\{u(s)\}_{s\in F}$ and $\{v(s)\}_{s\in F}$. Put $m=\max\{|F|,\|\lambda\|\}$ and  $k=1+\max\{\Lambda_{h(D)},\Xi_{m}\}$. Claim~\ref{cl2} guarantees that the element
$$y=\Sigma_\lambda(u)-\Sigma_\lambda(v)=\sum_{s\in F}\lambda(s)z_{u(s)}-\sum_{s\in F}\lambda(s)z_{v(s)}$$ has $y(k)\ne 0$, which implies that $y\ne 0$ and $\Sigma_\lambda(u)\ne\Sigma_\lambda(v)$.
\end{proof}

\begin{claim} The function $\Sigma_\lambda:V_F\to H$ is a topological embedding.
\end{claim}

\begin{proof} It suffices to show that a sequence $(v_n)_{n\in\w}\in V_F^\w$ contains a subsequence convergent to a point $v_\infty\in V_F$ if the sequence $(\Sigma_\lambda(v_n))_{n\in\w}$ converges to the point $\Sigma_\lambda(v_\infty)$ in $\IZ^\w$.
Replacing $(v_n)_{n\in\w}$ by a suitable subsequence, we can assume that for every $s\in F$ either the sequence $(v_n(s))_{n\in\w}$ converges to $v_\infty(s)$ or the point $v_\infty(s)$ has a neighborhood $W(v_\infty(s))\subset\w^\w$ containing no point of the set $\{v_n(s)\}_{n\in\w}$.
If the set $F_c=\{s\in F:\lim_{n\to\infty}v_n(s)=v_\infty(s)\}$ coincides with $F$, then the sequence $(v_n)_{n\in\w}$ converges to $v_\infty$ and we are done.

So, we assume that $F_c\ne F$ and hence the set $F_c'=F\setminus F_c$ is not empty. For every $s\in F'_c$ choose a number $m_s\ge\delta_F$ such that the basic neighborhood $V_{s|m_s}$ of $s$ is disjoint with the set $\{v_n(s)\}_{n\in\w}$. Then for every $n\in\w$ the set $D_n=\bigcup_{s\in F'_c}\{v_n(s),v_\infty(s)\}$ has $\delta_{D_n}\le\max_{s\in F}m_s$. Put $m=\max\{|F|,\|\lambda\|\}$ and  $k=1+\max\{\Xi_m,\max_{s\in F'_c}m_s\}$.

Consider the sequence $(v_n')_{n\in\w}\in\Pi_F^\w$ defined by
$$v_n(s)'=\begin{cases}
v_\infty(s)&\mbox{if $s\in F_c$}\\
v_n(s)&\mbox{if $s\in F'_c$}
\end{cases}
$$for every $n\in\w$. Claim~\ref{cl2} guarantees that for every $n\in\w$ the element
$y_n=\Sigma_\lambda(v'_n)-\Sigma_\lambda(v_\infty)$ has $y_n(k)\ne 0$
(here we also use the inequality $k\ge\delta_{D_n}$), which implies that $d\big(\Sigma_\lambda(v'_n),\Sigma_\lambda(v_n)\big)\ge 2^{-k}$.
On the other hand, the convergence of the sequence $(v_n(s))_{n\in\w}$ to $v_\infty(s)$ for $s\in F_c$ implies that $\lim_{n\to\infty}d(\Sigma_\lambda(v_n'),\Sigma(v_n))=0$. Taking into account the triangle inequality we conclude that the sequence $\big(d(\Sigma_\lambda(v_n),\Sigma_\lambda(v_\infty))\big)_{n\in\w}$ does not converge to zero, which is a desired contradiction completing the proof of the claim.
\end{proof}

\begin{claim}\label{cl:sp} The function $\Sigma_\lambda:V_F\to H$ is a closed topological embedding.
\end{claim}

\begin{proof} It suffices to prove that a sequence $(v_n)_{n\in\w}\in V_F^\w$ contains a convergent subsequence if the sequence $(\Sigma_\lambda(v_n))_{n\in\w}$ converges to some point $g\in H$.
Write $g$ as $g=\sum_{u\in E}\mu(u)z_u$ for some finite set $E\subset \w^\w$ and some function $\gamma:E\to\IZ\setminus\{0\}$.

Replacing $(v_n)_{n\in\w}$ by a suitable subsequence, we can assume that for every $u\in E$ either for some $s\in F$ the sequence $(v_n(s))_{n\in\w}$ converges to $u$ or the point $u$ has a neighborhood $W_u\subset\w^\w$ containing no point of the set $\{v_n(s):n\in\w,\;s\in F\}$. Consider the sets $E_c=\{u\in E:\exists s\in F\;\lim_{n\to\infty}v_n(s)=u\}$, $E_c'=E\setminus E_c$ and for every $u\in E_c$ choose a point $s_u\in F$ such that $u=\lim_{n\to\infty}v_n(s_u)$. Since the family $(V_{s|\delta_F})_{s\in F}$ is disjoint, the point $s_u$ is unique.
Let $F_c=\{s\in F:\exists u\in E\;\;s=s_u\}=\{s\in F:\exists u\in E$ with $u=\lim_{n\to\infty}v_n(s)\}$ and for every $s\in F_c$ put $u_s=\lim_{n\to\infty}v_n(s)\in V_{s|\delta_F}$.

If $F_c=F$, then the sequence $(v_n)_{n\in\w}$ is convergent in $V_F$ and we are done. So, we assume that $F_c\ne F$ and hence the set $F'_c=F\setminus F_c$ is not empty.
For every $u\in E'_c$ choose a number $m_u\ge\delta_{F\cup E}$ such that the neighborhood $V_{u|m_u}$ is disjoint with the set $\{v_n(s):s\in F,\;n\in\w\}$.

Consider the sequence $(v_n')_{n\in\w}\in\Pi_F^\w$ defined by
$$v_n'(s)=\begin{cases}
u_s&\mbox{if $s\in F_c$},\\
v_n(s)&\mbox{if $s\in F\setminus F_c$},
\end{cases}
$$for every $n\in\w$.

Let $m=\max\{|F|+|E|,\|\lambda\|+\|\mu\|\}$ and $k=1+\max\{\Xi_m,\max_{q\in E_c'}m_q\}$. Claim~\ref{cl2} ensures that the point $y=\Sigma_\lambda(v'_n)-\Sigma_\mu(u)$ has $y(k)\ne0$, which implies $d(\Sigma_\lambda(v'_n),g)=d(\Sigma_\lambda(v'_n),\Sigma_\mu(u))\ge 2^{-k}$. On the other hand, the convergence of the sequences $(v_n(s))_{n\in\w}$, $s\in F_c$, implies that $d(\Sigma_\lambda(v_n),\Sigma_\lambda(v'_n))\to 0$. Then the triangle inequality ensures that the sequence $\big(\Sigma_\lambda(v_n),g)_{n\in\w}$ does not converge to zero, contradicting with the choice of $g$.
\end{proof}

Claim~\ref{cl2} implies

\begin{claim}\label{cl:meg} The set $\Sigma_\lambda(V_F)\subset\{y\in\IZ^\w:\exists m\in\w\;\forall k\ge m\;\;y(k)\ne0\}$ is meager.
\end{claim}

Denote by $[\w^{<\w}]^{<\w}$ the family of finite subsets of $\w^{<\w}$. For the empty set $F=\emptyset$ and the unique map $\lambda\in \IZ_{\ne0}^F$ we put $\Sigma_\lambda(V_F)=\{0\}$. Claims~\ref{cl:sp}, \ref{cl:meg} and the obvious equality
$$H=\bigcup_{F\in[\w^{<\w}]^{<\w}}\bigcup_{\lambda\in \IZ_{\ne0}^F}\Sigma_\lambda(V_F)$$imply that the subgroup $H$ is $\sigma$-Polish and is contained in the meager subset $\{0\}\cup\bigcup_{m\in\w}\{z\in \IZ^\w:\allowbreak\forall k\ge m\;\;z(k)=0\}$ of $\IZ^\w$.
\end{proof}

\section{Set-Theoretic constructions of subgroups and the cardinal characteristic $\voc^*$}\label{s7}

In Example~\ref{e:counter} we constructed a subgroup $H\subset \IZ^\w$ which is meager and Haar null but does not belong to the $\sigma$-ideal $\E$ generated by closed Haar null subsets of $\IZ^\w$.
It is natural to ask if a subgroup with these properties exists in any Polish group, in particular in the Cantor cube $2^\w$, considered as the countable power of the two-element group $\{0,1\}$ endowed with the operation of addition modulo 2.

\begin{problem}\label{prob:ZFC} Is it true that the Cantor cube $2^\w$ contain a meager Haar null subgroup $G$ which cannot be covered by countably many closed Haar null sets? {\rm (By  Theorem~\ref{Lacz} such subgroup $G$ cannot be quasi-analytic).}
\end{problem}

In this section we shall prove that under Martin's Axiom the answer to this problem is affirmative.
More precisely, it is affirmative under the assumption $\cov(\N)=\cof(\M)=\cof(\M)$.

Problem~\ref{prob:ZFC} is a partial case of the following more general problem.

\begin{problem}\label{probmain2} Find conditions on two $\sigma$-ideals $\I,\J$ on a group $G$ under which the family $\I\setminus\J$ contains an subgroup of $G$.
\end{problem}

It turns out that even for the standard $\sigma$-ideals $\M$ and $\mathcal N$ on the Cantor cube $2^\w$ the answer to Problem~\ref{probmain2} is not trivial (and not symmetric). On one hand, for the family $\N\setminus\M$ we have the following ZFC-result due to Talagrand \cite{Tal}.

\begin{theorem}[Talagrand]\label{Tal} The compact topological group $2^\w$ contains a subgroup $H\subset 2^\w$ (which is an ideal on $\w$) that belongs to the family $\mathcal N\setminus \mathcal M$.
\end{theorem}

On the other hand, for the family $\M\setminus \N$ a counterpart of Talagrand's result is independent of ZFC, see \cite{Burke} or \cite{Paw}.

\begin{theorem}[Burke]\label{BurPaw} It is consistent that each meager subgroup of $2^\w$ is Haar null. On the other hand, Martin's Axiom implies the existence of a subgroup $H\in\M\setminus\N$ in $2^\w$. \end{theorem}

We shall give consistent answers to Problems~\ref{prob:ZFC} and \ref{probmain2} using some cardinal characteristics of ideals $\I,\J$ on a group $X$. In the sequel, we shall consider only ideals $\I$ on a set $X$ with the property: $\bigcup\I=X\notin\I$. In this case the following four cardinal characteristics are well-defined:
\begin{itemize}
\item $\add(\I)=\min\{|\A|: \A\subset \I \;\; \bigcup \A\not\in\I \}$;
\item $\non(\I)=\min\{|A|: A\subset X \;\; A\not\in\I \}$;
\item $\cov(\I)=\min\{|\A|: \A\subset \I \;\; \bigcup\A=X \}$;
\item $\cof(\I)=\min\{|\A|: \A\subset \I \;\; \forall I\in\I\; \exists A\in \A\; (I\subset A)\}$.
\end{itemize}
In fact, these four cardinal characteristics are partial cases of the following two cardinal characteristics defined for ideals $\I\subset\J$ on $X$:
\begin{itemize}
\item $\add(\I,\J)=\min\{|\A|: \A\subset \I \;\;\; \bigcup \A\not\in\J \}$;
\item $\cof(\I,\J)=\min\{|\A|: \A\subset \J \;\;\; \forall I\in\I\; \exists J\in \A\;\; (I\subset J)\}$;
\item $\cov(\I,\J)=\min\{|\A|:\A\subset\I \;\;\;G\setminus\bigcup\A\in\J\}$;
\end{itemize}
It is clear that $\add(\I)=\add(\I,\I)$, $\cof(\I)=\cof(\I,\I)$, $\cov(\I)=\cof(\F,\I)=\cov(\I,\I)$ and $\non(\I)=\add(\F,\I)$, where $\F$ is the ideal of finite subsets of $X$.

The following trivial proposition gives a simple answer to Problem~\ref{probmain2}.

\begin{proposition} Let $\I\subset\J$ be two ideals on a group $G$. If $\non(\I)<\non(\J)$, then there exists a subgroup $H\in\J\setminus\I$.
\end{proposition}

\begin{proof} By the definition of the cardinal $\non(\I)$, there is a subset $A\subset G$ of cardinality $|A|=\non(\I)$ such that $A\notin\I$. Then the subgroup $H$ generated by the set $A$ has cardinality $|H|\le\max\{\aleph_0,\non(\I)\}=\non(\I)<\non(\J)$ and hence $H$ belongs to the ideal $\J$.
\end{proof}

A more complicated answer to Problem~\ref{probmain2} will be given with help of the cardinal characteristic $\voc^*(\I,\J)$ defined for two ideals $\I,\J$ on a group $G$ as follows:
\begin{itemize}
\item $\cov^*(\I,\J)=\min\{|X|:X\subset G\mbox{ and }\exists I\in\I\;\;G\setminus (X\cdot I)\in\J\}$;
\item $\voc^*(\I,\J)=\sup\{\kappa:\exists I\in\I\; \forall X\subset G\;(|X|<\kappa\Rightarrow \bigcap_{x\in X}(xI)\notin\J)\}$;
\item $\non^*(\I)=\min\{|X|:X\subset G$ and $\forall g\in G\;\;g\cdot X\notin\I\}$;
\item $\cov^*(\I)=\cov^*(\I,\I)$;
\item $\voc^*(\I)=\voc^*(\I,\{\emptyset\})$.
\end{itemize}
This definition implies that $\voc^*(\I,\J)=1$ for any ideals $\I\subset\J$ on $X$.

An ideal $\I$ on a group $G$ is called {\em left-invariant} if for every $A\in\I$ and $g\in G$ the set $gA=\{ga:a\in A\}$ belongs to $\I$. It is easy to see that for any left-invariant ideals $\I\subset\J$ on a group $G$ we get $\cov(\I,\J)\le\cov^*(\I,\J)$. 

Given two ideals $\I,\J$ by $\I\vee\J$ we denote the ideal generated by the union $\I\cup\J$. For an ideal $\I$ on a group $G$ we put $\I^{-1}=\{A^{-1}:A\in\I\}$.

 \begin{proposition}\label{p5} Let $\I,\I'$ and $\J$ be two ideals on a group $G$ such that $G\in (\I\vee\I')\setminus\J$. Then $$\cov^*(\I',\J)\le\voc^*(\I,\J)\le \non^*(\I^{-1}).$$
 \end{proposition}

 \begin{proof} By the definition of the cardinal $\kappa=\cov^*(\I',\J)$, there exists a set $I'\in\I'$ such that $G\setminus (X\cdot I')\in\J$ for some set $X\subset G$ of cardinality $|X|=\cov^*(\I',\J)$. Since $G\in\I\vee\I'$, we can enlarge the set $I'$ and assume that $I:=G\setminus I'\in\I$. To show that $\kappa\le\voc^*(\I,\J)$ it suffices to show that for any subset $Y\subset G$ of cardinality $|Y|<\kappa$ we get $\bigcap_{y\in Y}yI\notin\J$. To derive a contradiction, assume that for some set $Y\subset G$ of cardinality $|Y|<\kappa$ the set $J=\bigcap_{y\in Y}yI$ belongs to the ideal $\J$. Then $G\setminus J=\bigcup_{y\in Y}y(G\setminus I)=\bigcup_{y\in Y}yI'$ and hence $G\setminus (Y\cdot I')\in\mathcal J$, which contradicts the definition of the cardinal $\cov^*(\I',J)>|Y|$.
\smallskip

 Next, we prove that $\voc^*(\I,\J)\le\non^*(\I^{-1})$. Assuming that $\voc^*(\I,\J)>\non^*(\I^{-1})$, we could find a set $I\in\I$ such that for any subset $X\subset G$ with $|X|\le\non^*(\I^{-1})$ we get $\bigcap_{x\in X}xI\notin\J$.  By the definition of the cardinal $\non^*(\I^{-1})$, there exists a set $X\subset G$ of cardinality $|X|=\non^*(\I^{-1})$ such that $gX\notin\I^{-1}$ for all $g\in G$. Then $X^{-1}g\notin\I$ and hence $X^{-1}g\not\subset I$ for all $g\in G$. We claim that $\bigcap_{x\in X}xI=\emptyset$. Indeed, for any point $g\in G$ we use $X^{-1}g\not\subset I$ to find a point $x\in X$ such that $x^{-1}g\notin I$ and hence $g\notin xI$.
 \end{proof}

A group $G$ is defined to be {\em Boolean} if $xxx=x$ for all $x\in G$. A typical example of a Boolean group is the Cantor cube $2^\w$. The following proposition (partially) answers Problem~\ref{probmain2}.

\begin{proposition}\label{p6} Let $\I\subset\J$ be two $\sigma$-ideals on a Boolean group $G$. If there exists an ideal $\I'$ on $G$ such that $\I\subset\I'$, $\J\not\subset\I'$ and $\cof(\I,\I')\le \voc^*(\J,\I')$, then the group $G$ contains a subgroup $H\in\J\setminus \I$.
\end{proposition}

\begin{proof} By the definition of $\voc^*(\J,\I')\ge\cof(\I,\I')$, there exists a set $A\in\J$ such that for any subset $X\subset A$ of cardinality $|X|<\cof(\I,\I')$ we get $\bigcap_{x\in X}xA\notin\I'$.
By the definition of the cardinal $\kappa=\cof(\I,\I')$, there is a subfamily $\{X_\alpha\}_{\alpha<\kappa}\subset \I'$ such that each set $I\in\I$ is contained in some set $X_\alpha$, $\alpha<\kappa$. By transfinite induction we shall construct a transfinite sequence of points $(x_\alpha)_{\alpha<\kappa}$ in $A$ such that for every $\alpha<\kappa$ the point $x_\alpha$ belongs to the set $\bigcap_{x\in H_{<\alpha}}xA\setminus X_\alpha$ where $H_{<\alpha}$ is the subgroup of $G$ generated by the set $\{x_\beta\}_{\beta<\alpha}$. It follows that for every $x\in H_{<\alpha}$ we get $x_\alpha\in xA$ and hence $x^{-1}x_\alpha=xx_\alpha\in A$. Consequently, $H_\alpha=H_{<\alpha}\cup x_\alpha H_{<\alpha}\subset A$. After completing the inductive construction, consider the subgroup $H$ generated by the set $\{x_\alpha\}_{\alpha<\kappa}$. It follows that $H\subset A$ and $H\not\subset X_\alpha$ for all $\alpha<\kappa$, which implies that $H\in\J\setminus\I$.
\end{proof}

For a locally compact topological group $G$ by $\M_G$ we shall denote the $\sigma$-ideal of meager sets in $G$, by $\mathcal N_G$ the $\sigma$-ideal of Haar null sets in $G$ and by $\mathcal E_G$ the $\sigma$-ideal generated by compact sets of Haar measure zero in $G$. If the group $G$ is clear from the context, they we shall omit the subscripts and shall denote these three $\sigma$-ideals by $\M,\N$ and $\E$. On the Cantor cube $2^\w$ the $\sigma$-ideal $\M,\N,\E$ were thoroughly studied in \cite{BaJu}. The cardinal characteristics of the ideals $\M,\N,\E$ on the Cantor cube $2^\w$ (or the real line) fit in the following extension of the famous C\'icho\'n diagram (in which an arrow $\kappa\to\lambda$ between two cardinals indicated that $\kappa\le\lambda$ in ZFC).
$$
\xymatrix{
&\cov(\N)\ar[r]&\non(\M)\ar@{=}[r]&\cof(\E,\N)\ar[r]&\cof(\E)\ar@{=}[r]&\cof(\M)\ar[r]&\cof(\N)\ar[r]
&\mathfrak c\\
&&\add(\E,\M)\ar[r]\ar[u]&\mathfrak b\ar[u]\ar[r]&\mathfrak d\ar[r]\ar[u]&\cof(\E,\M)\ar[u]\\
\w_1\ar[r]&\add(\N)\ar[r]\ar[uu]&\add(\M)\ar@{=}[r]\ar[u]&\add(\E)\ar[r]\ar[u]&\add(\E,\N)\ar@{=}[r]\ar[u]&
\cov(\M)\ar[r]\ar[u]&\non(\N)\ar[uu]
}$$
 All inequalities between the cardinal characteristics presented in this diagram are proved in Chapter 2 of \cite{BaJu} (more precisely, in Theorems 2.1.7, 2.2.9, 2.2.11, 2.3.7, 2.6.10, 2.6.14, 2.6.17 of \cite{BaJu}). It is well-known (see e.g. \cite{Vau} or \cite{Blass}) that Martin's Axiom collapses all cardinal characteristics of the ideals $\M,\N,\E$ appearing in this diagram to the continuum.

Applying Proposition~\ref{p6} to the ideals $\J,\I,\I'\in\{\M,\N,\M\cap\N,\E\}$, we get:

\begin{corollary}\label{c2} For the $\sigma$-ideals $\M,\N,\E$ on a Boolean locally compact topological group $G$ the following statements hold.
\begin{enumerate}
\item If $\cof(\mathcal N)\le\voc^*(\M,\N)$, then there exists a subgroup $H\in\M\setminus\N$.
\item If $\cof(\mathcal M)\le\voc^*(\N,\M)$, then there exists a subgroup $H\in\N\setminus\M$.
\item If $\cof(\E,\mathcal N)\le\voc^*(\M,\N)$, then there exists a subgroup $H\in\M\setminus\E$.
\item If $\cof(\E)\le\voc^*(\M\cap\N,\E)$, then there exists a subgroup $H\in(\M\cap\N)\setminus\E$.
\end{enumerate}
\end{corollary}

Corollary~\ref{c2} motivates the problem of calculating the cardinals $\voc^*(\I,\J)$ for the standard $\sigma$-ideals $\I,\J\in\{\M,\N,\M\cap\N,\E\}$. Proposition~\ref{p5} implies the following lower and upper bounds.

\begin{corollary}\label{c:voc*} For the ideals $\M,\N,\E$ on a locally compact topological group $G$ we get
\begin{enumerate}
\item $\cov(\N)\le\cov^*(\N)\le\voc^*(\M,\N)\le \voc^*(\M,\E)\le\voc^*(\M)\le\non^*(\M)=\non(\M)$;
\item $\cov(\M)\le\cov^*(\M)\le\voc^*(\N,\M)\le  \voc^*(\N,\E)\le\voc^*(\N)\le\non^*(\N)=\non(\N)$.
\end{enumerate}
\end{corollary}

Now we are going to evaluate the cardinal $\voc^*(\M\cap\N,\E)$.

\begin{lemma}\label{l:MNE} For locally compact topological groups $G,H$ we get
$$\voc^*(\M_{G\times H}\cap\N_{G\times H},\E_{G\times H})\ge \min\{\voc^*(\M_G,\E_G),\voc^*(\N_H,\E_H)\}.$$
\end{lemma}

\begin{proof} To derive a contradiction, assume that the cardinal $\kappa=\voc^*(\M_{G\times H}\cap\N_{G\times H},\E_{G\times H})$ is strictly smaller than $\min\{\voc^*(\M_G,\E_G),\voc^*(\N_H,\E_H)\}$. Let $\lambda_G$ and $\lambda_H$ be the Haar measures on the locally compact groups $G,H$, respectively. By a {\em Haar measure} on a locally compact group we understand any left-invariant Borel regular $\sigma$-additive measure that takes positive finite values on compact sets with non-empty interior. It is well-known that each locally compact topological group has a Haar measure and such measure is unique up to a positive multiplicative constant. This uniqueness result implies that the tensor product $\lambda=\lambda_G\otimes\lambda_H$ of the measures $\lambda_G$, $\lambda_H$ is a Haar measure on the group $G\times H$.

By the definition of the cardinals  $\voc^*(\M_G,\E_G),\voc^*(\N_H,\E_H)>\kappa$, there are sets $M\in\M_G$ and $N\in\N_H$ such that for any subsets $X\subset G$ and $Y\subset H$ with $\max\{|X|,|Y|\}\le\kappa$, we get $\bigcap_{x\in X}xM\notin\E_G$ and $\bigcap_{y\in Y}yN\notin\E_H$.
We claim that for any set $Z\subset G\times H$ of cardinality $|Z|\le\kappa$ we get $\bigcap_{z\in Z}z(M\times N)\notin\E_{G\times H}$. To derive a contradiction, assume that the intersection $\bigcap_{z\in Z}z(M\times N)$ belongs to the $\sigma$-ideal $\E_{G\times H}$ and hence can be covered by the union $\bigcup_{n\in\w}K_n$ of an increasing sequence $(K_n)_{n\in\w}$ of compact sets of Haar measure $\lambda(K_n)=0$ in $G\times H$. For every $n\in\w$ and point $x\in G$ consider the compact set $K_n^x=\{y\in H:(x,y)\in K_n\}$. The regularity of the measure $\lambda_H$ implies that for every $m\in\w$ the set $\{x\in G:\lambda_H(K_n^x)<2^{-m}\}$ is open in $G$. Then the set $P_{n,m}=\{x\in G:\lambda_H(K_n^x)\ge 2^{-m}\}$ is compact, being a closed subset of the projection of $K_n$ on $G$.
Let $X$ and $Y$ be the projections of the set $Z$ on $G$ and $H$, respectively.

We claim that $\bigcap_{x\in X}xM\subset \bigcup_{n,m\in\w}P_{n,m}$. Indeed, assuming that some point $g\in \bigcap_{x\in X}xM$ does not belong to $\bigcup_{n,m\in\w}P_{n,m}$, we conclude that $\lambda_H(K_n^g)=0$ for all $n\in\w$ and then
$$\{g\}\times\bigcap_{y\in Y}yN\subset(\{g\}\times H)\cap\bigcap_{z\in Z}z(M\times N)\subset
 (\{g\}\times H)\cap \bigcup_{n\in\w}K_n=\{g\}\times\bigcup_{n\in\w} K_n^g,$$
 which implies that the set $\bigcap_{y\in Y}yN\subset\bigcup_{n\in\w}K_n^g$ belongs to the $\sigma$-ideal $\E_H$. But this contradicts the choice of the set $N$. This contradiction shows that
 the set $\bigcap_{x\in X}xM\notin\E_G$ belongs to the $\sigma$-compact set $\bigcup_{n,m\in\w}P_{n,m}$, which implies that $\lambda_G(P_{n,m})>0$ for some $n,m\in\w$. Since for every $x\in P_{n,m}$ the set $K_n^x$ has measure $\lambda_H(K_n^x)\ge 2^{-m}$, the Fubini Theorem \cite[3.4.4]{Bog} guarantees that $\lambda(K_n)\ge 2^{-n}\cdot \lambda_G(P_{n,m})>0$. But this contradicts the choice of the set $K_n$.

 This contradiction shows that the set $M\times N\in\M_{G\times H}\cap\N_{G\times H}$ witnesses that $\voc^*(\M_{G\times H}\cap\N_{G\times H})>\kappa$, which contradicts the definition of the cardinal $\kappa$.
 \end{proof}

Lemma~\ref{l:MNE} will help us to evaluate the cardinal $\voc^*(\M\cap\N,\E)$ for any locally compact group $G$, isomorphic to its square $G\times G$.

\begin{corollary}\label{c:MiN} If a locally compact group $G$ is topologically isomorphic to its square $G\times G$, then $$\voc^*(\M\cap\N,\E)=\min\{\voc^*(\M,\E),\voc^*(\N,\E)\}.$$
 \end{corollary}

Combining Corollaries~\ref{c2}, \ref{c:voc*}, \ref{c:MiN} with the equalities $\cof(\mathcal E,\N)=\non(\M)$ and $\cof(\mathcal E)=\cof(\M)$ (proved in  Theorems 2.6.14 and 2.6.17 of \cite{BaJu}) we get another corollary.

\begin{corollary}\label{c4} The compact Boolean group $2^\w$ contains a subgroup
\begin{enumerate}
\item $H_1\in\M\setminus\N$ if $\cof(\N)=\cov(\N)$.
\item $H_2\in\M\setminus\mathcal E$ if $\non(\M)=\cov(\N)$.
\item $H_3\in(\M\cap\mathcal N)\setminus\mathcal E$ if $\cof(\M)=\cov(\M)=\cov(\N)$.
\end{enumerate}
\end{corollary}

\begin{remark}
By Theorem~\ref{BurPaw}, it is consistent that each meager subgroup of $2^\w$ belongs to the ideal $\mathcal N$. So, the first statement of Corollary~\ref{c4} cannot be proved in ZFC. We do not know if the same is true for the other two statements of Corollary~\ref{c4}.
\end{remark}

\begin{problem} Calculate the values of the cardinal characteristics $\voc^*(\M,\I)$ and $\voc^*(\N,\I)$ for the $\sigma$-ideals $\I\in\{\F,\E,\mathcal M\cap\N,\M,\N\}$ on the group $2^\w$. Are they equal to some known cardinal characteristics of the continuum?
\end{problem}

Corollary~\ref{c:voc*} combined with known information \cite{BaJu} about the cardinal characteristics of the $\sigma$-ideals $\M,\N,\E$ of the Cantor cube $2^\w$ allow us to draw the following diagram.
{\small
$$\xymatrix{
\cov^*(\M)\ar[r]&\voc^{\!*}\!(\N\!,\M)\ar[r]&\voc^*\!(\N\!,\E)\ar[r]&\voc^*(\N)\ar[r]&\non^*(\N)\ar@{=}[r]&\non(\N)\ar[r]&\cof(\N)\\
\cov(\M)\ar@{=}[r]\ar[u]&\add(\E,\N)\ar[rr]\ar[u]&&\mathfrak d\ar[r]&\cof(\E,\M)\ar[r]&\cof(\E)\ar@{=}[r]\ar[u]&\cof(\M)\ar[u]\\
\add(\M)\ar@{=}[r]\ar[u]&\add(\E)\ar[r]\ar[u]&\add(\E\!,\!\M)\ar[r]&\mathfrak b\ar[rr]\ar[u]&&\cof(\E,\N)\ar@{=}[r]\ar[u]&\non(\M)\ar[u]\\
\add(\M)\ar[r]\ar[u]&\cov(\N)\ar[r]\ar[u]&\cov^*(\N)\ar[r]&\voc^*(\M,\N)\ar[r]&\voc^{\!*}\!(\M,\E)\ar[r]& \voc^*(\M)\ar[r]\ar[u]&\non^{\!*}(\M)\ar@{=}[u]
}$$
}

\begin{remark} The cardinal characteristic $\cov^*(\I)$ of various $\sigma$-ideals $\I$ on Polish groups have been  studied in \cite[\S2.7]{BaJu}, \cite{MS}, \cite{DM}, \cite{ES}, \cite{ET}. By \cite{DM}, for a non-locally compact Polish group $G$ possessing an invariant metric we get $\cov^*(\M)=\cov(\M)$. On the other hand, it is consistent \cite{MS} that for some locally compact groups $G$ (like $\IR$, $\IR/\IZ$ or $2^\w$) we get $\cov^*(\M)>\cov(\M)$. This implies the consistence of a strict inequality $\cov(\M)<\voc^*(\M,\N)$. By \cite{ET}, each Polish locally compact abelian group $G$ has $\cov^*(\E)\le \cof(\N)$ and it is consistent that $\cov^*(\E)<\cof(\N)$.
\end{remark}

For the Polish group $\IZ^\w$ the cardinal $\voc^*(\M)$ can be easily calculated using the combinatorial characterization of the cardinal $\non(\M)$.

\begin{proposition} For the Polish group $G=\IZ^\w$ we get $\voc^*(\M)=\non(\M)$.
\end{proposition}

\begin{proof} In the Polish group $G$ consider the meager subset $F=\{x\in\IZ^\w:|\{n\in\w:x(n)=0\}|<\w\}$. We claim that for any subset $X\subset G$ of cardinality $|X|<\non(\M)$ we get $\bigcap_{x\in X}xF\ne\emptyset$. Indeed, by Theorem 2.4.7 \cite{BaJu}, there exists a function $y\in\IZ^\w$ such that for every $x\in X$ the set $\{n\in\w:x(n)=y(n)\}$ is finite and hence $y\in xF$. This means that the intersection $\bigcap_{x\in X}xF\ni y$ is not empty and hence $\voc^*(\M)\ge\non(\M)$. The inequality $\voc^*(\M)\le\non(\M)$ follows from Proposition~\ref{p5}.
\end{proof}

\begin{problem} Is $\voc^*(\M)=\non(\M)$ for the ideal $\M$ of meager sets on the Cantor cube $2^\w$?
\end{problem}

\end{document}